\documentclass[12pt,a4paper,reqno]{amsart}
\usepackage[T1]{fontenc}
\usepackage[marginratio=1:1,totalwidth=15.75cm,totalheight=22.275cm]{geometry}
\usepackage[utf8]{inputenc}
\usepackage{amsmath,amssymb,amsthm}
\usepackage[alphabetic,nobysame]{amsrefs}

\numberwithin{equation}{section}

\usepackage[shortlabels]{enumitem}
\usepackage[symbol,perpage]{footmisc}
\usepackage{hyperref}
\usepackage{color}

\usepackage{fixdif}
\newdif*{\Deriv}{\mathrm{D}}

\newcommand{\deriv}{\d}

\newcommand{\im}{\operatorname{Im}}

\newcommand{\discF}{\tilde{\mathcal{F}}}
\newcommand{\discL}{\tilde{\mathcal{L}}}
\newcommand{\downF}{\mathcal{F}}
\newcommand{\downL}{\mathcal{L}}
\newcommand{\limitF}{\mathcal{F}}
\newcommand{\limitL}{\mathcal{L}}

\renewcommand{\phi}{\varphi}

\newcommand{\cl}{\operatorname{cl}}
\newcommand{\interior}{\operatorname{int}}

\newcommand{\eps}{\varepsilon}

\newcommand*{\defeq}{\mathrel{\vcenter{\baselineskip0.5ex \lineskiplimit0pt
                     \hbox{\scriptsize.}\hbox{\scriptsize.}}}%
                     =}
\newcommand{\eqdef}{=\mathrel{\vcenter{\baselineskip0.5ex \lineskiplimit0pt
                     \hbox{\scriptsize.}\hbox{\scriptsize.}}}}

\newcommand{\RR}{\mathbb{R}}
\def\Cx{\mathbb{C}}
\def\Chat{\widehat{\mathbb{C}}}

\def\diam{\operatorname{diam}}
\def\intr{\operatorname{int}}
\def\dist{\operatorname{dist}}
\newcommand{\Aut}{\operatorname{Aut}}
\newcommand{\id}{\operatorname{id}}

\DeclareMathOperator{\Mod}{mod}
\DeclareMathOperator{\inj}{inj}

\newcommand{\DD}{\mathbb{D}}
\DeclareMathOperator{\out}{out}

\theoremstyle{plain}
\newtheorem{thm}{Theorem}[section]
\newtheorem{lemma}[thm]{Lemma}
\newtheorem{cor}[thm]{Corollary}
\newtheorem{prop}[thm]{Proposition}
\newtheorem{thmout}{Theorem}

\theoremstyle{definition}
\newtheorem{defi}[thm]{Definition}

\newtheorem*{rmk}{Remark}

\usepackage{orcidlink}

\title{Classifying multiply connected wandering domains}
\author[G. R. Ferreira]{Gustavo R. Ferreira\,\orcidlink{0000-0002-7330-0018}}
\address{Centre de Recerca Matem\`atica, Barcelona, Spain}
\email{grodrigues@crm.cat}
\author[L. Rempe]{Lasse Rempe\,\orcidlink{0000-0001-8032-8580}}
\address{Department of Mathematics\\ University of Manchester\\ Manchester M13 9PL\\ United Kingdom} 
\email{lasse.rempe@manchester.ac.uk}

\thanks{This research was partly conducted when both authors worked at the University of Liverpool; support by the University for this research, and in particular for funding for a post-doctoral position of the first author, is gratefully acknowledged. The first author also acknowledges financial support from the Spanish State Research Agency through the Mar\'ia de Maeztu Program for Centers and Units of Excellence in R\&D (CEX2020-001084-M) and from the European Commission Horizon Europe research and innovation programme under the Marie Skłodowska-Curie Grant Agreement No. 101208256. Finally, the authors would like to thank the Isaac Newton Institute for Mathematical Sciences, Cambridge, for support and hospitality during the programme \emph{Complex dynamics: a century on from Fatou and Julia}, where work on this paper was undertaken. This work was supported by EPSRC grant EP/Z000580/1.}

\date{\today}

\newcommand{\lr}[1]{{#1}}
\begin{document}
\begin{abstract}
We study the internal dynamics of multiply connected wandering domains of meromorphic functions. We do so by considering the sequence of injectivity radii along the orbit of a base point, together with the hyperbolic distortions along the same orbit. The latter sequence has previously been used in this context; the former introduces geometric information about the shape of the wandering domains that interacts with the dynamical information given by the hyperbolic distortions. Using this idea, we complete the description of the internal dynamics of any wandering domain of a meromorphic function, and also unify previous approaches to the question. We conclude that the internal dynamics of a wandering domain, from the point of view of hyperbolic geometry, can be classified into six different types. Five of these types were previously known and are realised by wandering domains of entire functions. The sixth type  arises only for meromorphic functions: a locally but not globally eventually isometric wandering domain. We construct a meromorphic function with such a domain, demonstrating that this new phenomenon does in fact occur. Our results show that, on a local level, the  dynamical behaviour of wandering domains of meromorphic functions is similar to that which occurs \lr{for} entire functions. However, new global phenomena can arise in the non-entire case. 

In particular, we identify a category of multiply connected wandering domains~-- essentially thin and infinitesimally non-contracting domains~-- that naturally generalise multiply connected wandering domains of entire functions, and share many characteristics with these. We generalise a number of results of Bergweiler, Rippon and Stallard for the entire case to this more general setting. In particular, we show the existence of dynamically meaningful singular foliations provided by level lines of certain positive harmonic functions. 

The results of this paper have been formalised and verified in the 
Lean 4 proof assistant.
 \end{abstract}
\maketitle

\section{Introduction}
Let $f\colon\Cx\to\Chat$ be a meromorphic function, where $\Chat = \Cx\cup\{\infty\}$ denotes the Riemann sphere. The \textit{Fatou set} of $f$, denoted $F(f)$, is defined as the set of points $z\in\Cx$ such that, on some neighbourhood of $z$, the iterates $(f^n)_{n=0}^{\infty}$ are defined and form a normal family in the sense of Montel. In other words, it is the set where the dynamics is not sensitive to initial conditions -- nearby points have similar long-term behaviour. The complement of the Fatou set, the \emph{Julia set} $J(f)$, is the locus of ``chaotic'' dynamics. Both sets are completely invariant under $f$, meaning that $f^{-1}(F(f))\subset F(f)$ and $f(F(f))\subset F(f)$, and analogously for the Julia set (with special care taken for the point at infinity). In particular, if $U$ is a connected component of the Fatou set~-- henceforth a \textit{Fatou component}~-- then for all $n\geq 0$, $f^n(U)$ is contained in some Fatou component $U_n$ (where $U_0 = U$). If $U_n \neq U_m$ for all $n\neq m$, we say that $U$ is a \textit{wandering domain}.

The dynamical behaviour of wandering domains of meromorphic functions is an area of intense research. Recent efforts \cites{BEFRS19,Fer22,Fer24} to classify their internal dynamics have focused on the behaviour of the hyperbolic metric -- something intrinsic to the domains themselves. More precisely, let $f\colon\Cx\to\Chat$ be a meromorphic function with a wandering domain $U$, and let $z$ and $w$ be distinct points in $U$. The primary question of this line of research has been: what is the long-term behaviour of the sequence
\begin{equation}\label{eqn:distancesequence} \left(d_{U_n}\left(f^n(z), f^n(w)\right)\right)_{n=0}^{\infty}, \end{equation}
where $d_\Omega$ denotes the hyperbolic distance in the domain $\Omega\subset\Chat$? By the Schwarz--Pick theorem (see Theorem \ref{thm:SP}), this sequence 
of non-negative numbers is non-increasing, and hence tends to a limit $c(z, w)$. The next questions, then, are whether this limit is zero or positive, and whether it is achieved. These were analysed by Benini \textit{et al.} \cite{BEFRS19} in the case when all the domains $U_n$ are simply connected, and by the first author \cites{Fer22,Fer24} for several cases of multiply connected wandering domains. The purpose of this paper is to provide a framework that both unifies the existing progress and completes the classification of wandering domains, enabling us to tackle previously unaddressed types of wandering domains. 

Let $U$ be a wandering domain of a meromorphic function $f$, and let $z\in U$. The sequence of \textit{hyperbolic distortions} of $f$ along the orbit of $z$ is the sequence $(\lambda_n(z))_{n=0}^{\infty}$, where
\begin{equation}\label{eq:distortionseq}
    \lambda_n(z) = \|\Deriv f\left(f^{n}(z)\right)\|_{U_{n}}^{U_{n+1}};
\end{equation}
see~\eqref{eqn:hypdistortion} for the definition of hyperbolic distortion. This object was first introduced in \cite{BEFRS19} in the context of simply connected wandering domains, where it plays an important role in the study of internal dynamics. It will be important for us as well, and we start by showing that it behaves in the same way at all points in the domain. This claim can be deduced from the results of \cite{BEFRS19} and \cite{Fer24}, but we give here a direct proof of a stronger result that does not rely on previous knowledge of internal dynamics of wandering domains.

\begin{prop}\label{prop:trich}
Let $f\colon\Cx\to\Chat$ be a meromorphic function with a wandering domain $U$. Then, for any $z, w\in U$ and all $n\geq 0$,
\[ e^{-2d_U(z,w)} (1 - \lambda_n(w)) \leq 1 - \lambda_n(z) \leq 
e^{2d_U(z,w)}(1 - \lambda_n(w)). \]
In particular, whether the series
\[ \sum_{n=0}^{\infty} (1 - \lambda_n(z)) \]
converges, diverges, or has a finite number of non-zero terms does not depend on the choice of $z\in U$.
\end{prop}

This means that every wandering domain exhibits one of three types of 
behaviour:
\begin{defi}\label{def:inf}
Let $f\colon\Cx\to\Chat$ be a meromorphic function with a wandering domain~$U$.
\begin{enumerate}[(i)]
    \item If $\sum_{n=0}^{\infty}\left(1 - \lambda_n(z)\right) = +\infty$ for all $z\in U$, we say that $U$ is \emph{infinitesimally contracting}.
    \item If $\sum_{n=0}^{\infty}\left(1 - \lambda_n(z)\right) < +\infty$ and $\lambda_n(z) < 1$ for infinitely many values of $n$ for every $z\in U$, we say that $U$ is \emph{infinitesimally semi-contracting}.
    \item If there exists $N\geq 0$ such that $\lambda_n(z) = 1$ for all $n\geq N$ and every $z\in U$, we say that $U$ is \emph{infinitesimally eventually isometric}.
\end{enumerate}
\end{defi}

Let us take a moment to discuss the significance of the series 
  \begin{equation}\label{eqn:distortionsum}
     \sum_{n=0}^{\infty}\left(1 - \lambda_n(z)\right).\end{equation}
 For any $z\in U$ and $k\geq 1$, the chain rule implies that
\begin{equation}\label{eqn:chainrule} \|\Deriv f^k(z)\|_U^{U_k} = \prod_{n=1}^k \|\Deriv f\left(f^{n-1}(z)\right)\|_{U_{n-1}}^{U_n}. \end{equation}
Since the hyperbolic distortion is always in the interval $[0, 1]$, the product on the right-hand side necessarily converges to a non-negative number as $k\to+\infty$. If $f^n(z)$ is not a critical point of $f$ for any $n\geq 0$, then this limit is non-zero if and only if the sum $\sum_{n\geq \lr{0}}\left(1 - \lambda_n(z)\right)$ is convergent. So the classification in Definition~\ref{def:inf} is really about the behaviour
of the hyperbolic derivative $\| \Deriv f^k(z)\|_U^{U_k}$ for all but a countable set of $z$: if $U$ is infinitesimally contracting, this \lr{sequence} tends to $0$\lr{;
 if $U$} is infinitesimally eventually isometric, the sequence is eventually constant\lr{. In} the infinitesimally semi-contracting case, \lr{the sequence} converges to, but is not eventually  equal to, a non-zero limit. Stating Proposition~\ref{prop:trich} and Definition \ref{def:inf} in terms of  
 the sum~\eqref{eqn:distortionsum} allows us to state the definition
  without having to exclude iterated preimages of critical points.

Using the terminology from Definition~\ref{def:inf}, we can restate a key result from \cite{BEFRS19} as follows.
\begin{thmout}[\cite{BEFRS19}, Theorem A]\label{thmout:sc}
Let $f\colon\Cx\to\Cx$ be an entire function with a simply connected wandering domain $U$. Define the set
 \[ E \defeq \{(z, w)\in U\times U\colon \text{$f^k(z) = f^k(w)$ for some $k\geq 0$}\}. \]
\begin{enumerate}[(i)]
    \item If $U$ is infinitesimally contracting, then $d_{U_n}\left(f^n(z), f^n(w)\right)\to 0$ for all $z$ and $w$ in $U$;\label{item:BEFRScontr}
    \item If $U$ is infinitesimally semi-contracting, then
      for all $(z, w)\in (U\times U)\setminus E$, there is $c(z,w)>0$ such that
        $d_{U_n}\left(f^n(z), f^n(w)\right)\to c(z, w) > 0$ but $d_{U_n}\left(f^n(z), f^n(w)\right)\neq c(z, w)$ for all $n$;
    \item If $U$ is infinitesimally eventually isometric, then there exists $N\geq 0$ such that, for all $(z,w)\in (U\times U)\setminus E$, there is $c(z,w)>0$ with
      $d_{U_n}\left(f^n(z), f^n(w)\right) = c(z, w)$ for $n\geq N$.\label{item:BEFRSisometric}
\end{enumerate}
\end{thmout}

Benini \textit{et al.} called the behaviours described in Theorem \ref{thmout:sc}\ref{item:BEFRScontr}-\ref{item:BEFRSisometric} \textit{contracting}, \textit{semi-contracting}, and \textit{eventually isometric} (respectively). When we wish to emphasise the distinction with the notions from Definition~\ref{def:inf}, we shall refer to this as \emph{globally} contracting, semi-contracting, or eventually isometric behaviour. Theorem~\ref{thmout:sc} also holds for simply connected wandering domains of meromorphic functions, as long as one assumes that all iterates of the domain are simply connected. From  a hyperbolic-geometric perspective, this completely describes the internal dynamics of simply connected wandering domains: \textit{``if a simply connected wandering domain $U$ has only simply connected iterates and is infinitesimally contracting (resp.\ semi-contracting, eventually isometric), then it is globally contracting (resp.\ semi-contracting, eventually isometric)''}.

Because of the complex interplay between the mapping properties of the function and the topology and geometry of the domains, the above maxim fails for multiply connected wandering domains. Indeed, it was observed in \cite{Fer22} that such a wandering domain may also be~\emph{bimodal}~-- i.e., there are pairs $(z,w)\notin E$ such that 
 the sequence~\eqref{eqn:distancesequence} converges to $0$, while for others it tends nontrivially to a non-zero limit~-- or \emph{trimodal}~-- the sequence
 may exhibit all three of the behaviours listed in Theorem~\ref{thmout:sc}. In order to provide a unified explanation for the internal dynamics of wandering domains, we introduce the following definition.

\begin{defi}\label{def:thickthin}
Let $f$ be a meromorphic function with a wandering domain $U$. For $z\in U$ and $n\geq 0$, let $\delta_n(z)$ denote the injectivity radius of $U_n$ at $f^n(z)$ (see Definition~\ref{def:inj}). We say that $U$ is \textit{essentially thin} if there exists $z_0\in U$ and an increasing sequence $(n_k)_{k=0}^{\infty}$ such that $\delta_{n_k}(z_0)\to 0$ as $k\to +\infty$. A wandering domain that is not essentially thin is said to be \textit{essentially thick}.
\end{defi}

Let us reassure ourselves that this definition does not depend on the choice of $z_0\in U$.

\begin{prop}\label{prop:injglobal}
Let $f$ be a meromorphic function with a wandering domain $U$, and let $z_0\in U$. If $(n_k)_{k=0}^{\infty}$ is a sequence such that $\delta_{n_k}(z_0)\to 0$, then $\delta_{n_k}(z)\to 0$ for all $z\in U$.
\end{prop}

Propositions \ref{prop:trich} and \ref{prop:injglobal} show that the qualitative behaviour of the sequences
\[ \text{$\left(\lambda_n(z)\right)_{n=0}^{\infty}$ and $\left(\delta_n(z)\right)_{n=0}^{\infty}$} \]
does not depend on our choice of $z\in U$. Taken together, the asymptotic behaviour of these two sequences shall give us almost complete information about the internal dynamics of the wandering domain, giving rise to a six-way classification. Five of these six cases were previously known, while the sixth is new. Before we state the classification theorem (Theorem~\ref{thm:main}) formally, we introduce some further terminology.

We say that a wandering domain $U$ has \emph{eventual connectivity} \lr{$c$}, where $1\leq \lr{c} \leq \infty$, if $U_n$ has connectivity $\lr{c}$ for all sufficiently large $n$. Note that, if $U$ has eventual connectivity $1$, then it is essentially thick by definition; in this case, we call $U$ \emph{eventually simply connected}. We also say that $U$ is ``locally eventually isometric'' if condition~\ref{item:BEFRSisometric} of Theorem~\ref{thmout:sc} holds for every compact subset $K\subset U$. That is, there exists $N = N(K)\geq 0$ such that the sequence~\eqref{eqn:distancesequence} is
   constant for $n\geq N$, when $(z, w)\in \lr{(K\times K)}\setminus E$. 
   (It was shown in~\cite{Fer24}*{Theorem 1.1} that it is equivalent to require only that the condition holds for \emph{some} $K$ with non-empty interior.)

\begin{thm}\label{thm:main} Let $U$ be a wandering domain of a transcendental meromorphic function. Then
   one of the following holds.
  \begin{enumerate}[(1)]
    \item \label{item:contracting}
     $U$ is globally contracting. 
    $U$ has eventual connectivity $1$, \lr{$2$,} $\infty$, or
       none. 
   \item $U$ is essentially thick and globally semi-contracting.  $U$ has eventual connectivity $1$, \lr{$2$,} $\infty$, or
       none.\label{item:semicontracting}
    \item  $U$ is essentially thin, infinitesimally semi-contracting, and bimodal. $U$ has eventual connectivity \lr{$2$ or} $\infty$ and $\delta_n(z)\to 0$ for all $z\in U$.\label{item:bimodal}
    \item $U$ is essentially thick and globally eventually isometric. $U$ has finite or infinite eventual connectivity.\label{item:globallyisometric}
    \item $U$ is essentially thick and locally, but not globally, eventually isometric. $U$ has eventual connectivity $\infty$.\label{item:nongloballyisometric}
    \item $U$ is essentially thin, infinitesimally eventually isometric and trimodal. $U$ has eventual connectivity $2$ and $\delta_n(z)\to 0$ for all $z\in U$.\label{item:trimodal}
  \end{enumerate}
\end{thm}

With the exception of~\ref{item:nongloballyisometric}, \lr{each of the cases~\ref{item:contracting}--\ref{item:trimodal}} may be realised using known constructions\lr{, even with entire $f$.} \lr{(If $f$ is entire in Theorem~\ref{thm:main}, all $U_n$ must be simply connected in cases~\ref{item:contracting}, \ref{item:semicontracting} and~\ref{item:globallyisometric}, and
infinitely connected in case~\ref{item:bimodal}. See} Section~\ref{sec:ex} for further details\lr{, including a discussion of which eventual
connectivities are known to arise for meromorphic functions.)} 
\lr{On the other hand, for entire $f$ and multiply connected $U$,} $\delta_n(z_0)\to 0$. \lr{In} particular,~\ref{item:nongloballyisometric} cannot occur in this case.     
It is natural to ask whether~\ref{item:nongloballyisometric} can be realised for non-entire meromorphic functions. Our next theorem shows that this is indeed the case.
\begin{thm} \label{thm:ex}
There exists a transcendental meromorphic function $g$ with a wandering domain that is locally eventually isometric but not globally eventually isometric.
\end{thm}

In the process of proving Theorem~\ref{thm:main}, we show that essentially thin wandering domains may exhibit striking similarities to multiply connected wandering domains of entire functions. Let us take a moment to explore this similarity in greater detail.

\subsection*{Singular foliations in essentially thin domains}
Suppose that $U$ is a multiply connected wandering domain of an entire function; then $U$ is essentially thin and non-contracting (i.e., either infinitesimally eventually isometric or infinitesimally semi-con\-tracting). It is known that there exists an absorbing sequence of round annuli $A_n\subset U_n$ such that any orbit in $U$ enters $A_n$ for all sufficiently large $n$~\cite{BRS13}*{Theorem~1.3}. Moreover, $\Mod A_n\to+\infty$ as $n\to+\infty$. (Recall that the \emph{modulus} of a doubly connected hyperbolic Riemann surface $A$ is the unique value of $M \in(0,\infty]$ such that $A$ is biholomorphic to $\{z\colon 1 < |z| < e^{2\pi M}\}$; see \cite{Hub06}*{Proposition 3.2.1}.) Furthermore, there exist dynamically defined \lr{singular foliations} that determine the long-term behaviour of the sequence~\eqref{eqn:distancesequence}. Here a \lr{\emph{singular foliation}} of a domain $U$ is a \lr{partition} of $U$ into sets, called \emph{leaves}, \lr{such that these leaves can locally be straightened either to the foliation of 
a square by horizontal leaves, or to an $m$-pronged singularity (where $m\geq 1$, $m\neq 2$).
(See 
Section~\ref{sec:prelim} for a formal definition, and compare also } 
\cite{Fer22}*{pp.~1899--1900}, and \cite{Fer22}*{p.~1902}.) The same turns out to be true for essentially thin and non-contracting cases in general, i.e.\ for the bi- and trimodal cases in Theorem~\ref{thm:main}.
  
\begin{defi}\label{def:foliations}
Let $U$ be a wandering \lr{Fatou component} of the meromorphic function~$f$.
\begin{itemize}
    \item We say that $U$ admits a \textit{contracting (singular) foliation} if there exists a singular foliation $\mathcal{C}$ of $U$ such that $d_{U_n}\left(f^n(z), f^n(w)\right)\to 0$ for every $z$ and $w$ on the same leaf of $\mathcal{C}$.
    \item We say that $U$ admits an \textit{eventually isometric (singular) foliation} if there exists a singular foliation $\mathcal{L}$ of $U$ \lr{and $n_0\geq 0$} such that, 
    for \lr{any two distinct points} $z$ and $w$ on the same leaf of \lr{$\mathcal{L}$} 
    \lr{with $f^{n_0}(z)\neq f^{n_0}(w)$} and \lr{for all $n\geq n_0$}, $d_{U_n}\left(f^n(z), f^n(w)\right) = c(z, w) > 0$.\footnote{%
    \lr{The definition in~\cite{Fer22}*{p.~1900} omits $n_0$ and 
     the requirement that $f^{n_0}(z)\neq f^{n_0}(w)$. This was an 
    oversight; the possibility that two points on the same leaf have the
    same image after a bounded number of steps has to be allowed also in~\cite{Fer22}*{Theorem~1.1}.
    See Figure~\ref{fig:onion}, where $n_0=1$.}}
    \item We say that a singular foliation \lr{$\mathcal{F}_0$} on $U$ is
       \emph{dynamical} if \lr{there exists a family $(\mathcal{F}_n)_{n=0}^{\infty}$ such that
       each $\mathcal{F}_n$ is a singular foliation of 
       the Fatou component containing $f^n(U)$ and $f$ maps leaves of $\mathcal{F}_n$
       into leaves of $\mathcal{F}_{n+1}$}
       for every $n\geq 0$.
    \item \lr{A singular foliation on $U$ is \emph{harmonic} if its 
      local leaves are given by the path-connected components of level sets of a harmonic function.}  
\end{itemize}
\end{defi}

With this notion, we can make cases~\ref{item:trimodal} and~\ref{item:bimodal}
 of Theorem~\ref{thm:main} more precise as follows. \lr{(See Figure~\ref{fig:onion}.)}

\begin{figure}
\begin{center}
    \def\svgwidth{\textwidth}
    \input{onion-five-panel.pdf_tex}
\end{center}
\caption{\lr{Eventually isometric singular foliations for an orbit of essentially thin and
 infinitesimally eventually isometric wandering domains of a meromorphic function $f$. The foliation in $U_0$ has $4$-pronged
  singularities at simple critical points of $f$. The picture is generated using separate holomorphic
  maps on each domain $U_n$, but could be realised (up to arbitrarily small perturbations)
  by a global transcendental meromorphic function using approximation theory.}\label{fig:onion}}
\end{figure}

\begin{thm}\label{thm:EI}
Let $f$ be a meromorphic function, and let $U$ be an essentially thin infinitesimally eventually isometric wandering domain of $f$. Then $U$ has eventual connectivity $2$ and $\delta_n(z)\to 0$ for all $z\in U$; in particular, $\Mod U_n\to \infty$ as $n\to\infty$. Furthermore, $U$ admits a unique contracting singular foliation and a unique eventually isometric singular foliation\lr{.
Both foliations are dynamical and harmonic.}
\end{thm}

\begin{thm}\label{thm:semi-contracting}
Let $f$ be a meromorphic function with an essentially thin infinitesimally semi-contracting wandering domain $U$. Then $\delta_n(z)\to 0$ for all $z\in U$, and there \lr{are
$n_0\geq 0$ and} a sequence \lr{$(A_n)_{n=n_0}^{\infty}$} of homotopically non-trivial \lr{open} topological annuli $A_n\subset U_n$ such that $f(A_n)\subset A_{n+1}$ for all \lr{$n\geq n_0$}, $f|_{A_n}$ \lr{is a covering map onto its image}, $\Mod A_n\to \infty$ as $n\to+\infty$, and for any compact subset $K\subset U$, there exists $N\geq 0$ such that $f^n(K)\subset A_n$ for $n\geq N$. Furthermore, $U$ admits a unique contracting singular foliation, which is \lr{dynamical and harmonic}.
\end{thm}

Recall that any doubly connected plane domain of sufficiently large modulus contains a round annulus of not much smaller modulus. (This was first proved by Teichm\"uller; see~\cite{Ahl73}*{Equations (4-21) and (4-23)} and also~\cite{AW2009}*{Theorem~3.17}.) Furthermore, an unbranched covering map between topological annuli acts as a power map between annuli (see, for instance, \cite{BM06}*{Theorem~13.6}). With this in mind, Theorems~\ref{thm:EI} and~\ref{thm:semi-contracting} generalise results of Zheng \cite{Zhe06} on the existence of large round annuli in multiply connected wandering domains of meromorphic functions with finitely many poles, as well as results of Bergweiler, Rippon, and Stallard \cite{BRS13}*{Theorem 1.4} on absorbing annuli without critical points. Note that, in \cite{Zhe06} and \cite{BRS13}, the wandering domains themselves are large in the Euclidean sense, and so are the relevant subannuli. For meromorphic maps, the wandering domains may be small, in which case they contain subannuli that are small in the Euclidean sense, but with large modulus. See also \cite{Zhe17} for more results on the size of annuli contained in multiply connected wandering domains.

The above results suggest that essentially thin non-contracting wandering domains are natural generalisations of multiply connected wandering domains of entire functions. For further comparisons between our results and those of~\cite{BRS13}, see Section~\ref{sec:entire}.

\subsection*{The contracting case}
The case of an infinitesimally contracting wandering domain is the simplest from the point of view of the study of the sequence~\eqref{eqn:distancesequence}: such domains are always globally contracting, i.e.\ for every compact set $K\subset U_0$, the sequence tends to zero uniformly in $z,w\in K$ as $n\to\infty$. In other respects, the contracting case is the most flexible. For example, it is the only case where the sequence $(\delta_n(z))_{n=0}^{\infty}$ of injectivity radii may tend to zero on a subsequence, but not globally. Indeed, in the non-contracting case, we have either $\delta_n(z)\to \infty$, or $\delta_n$ is eventually non-increasing (see Lemma~\ref{lem:homomorphisms}). The flexibility of contracting wandering domains can also be seen in the study of \emph{grand orbit relations}. Following~\cite{nuriachristian} 
 and~\cite{vassoetalorbitrelations}, we say that a point $z_0\in U_0$ has
 \emph{discrete grand orbit relation} if the intersection of the grand orbit of $z_0$ under $f$
 with $U_0$~-- that is, the set of $z$ with $f^n(z)=f^n(z_0)$ for some $n\geq 0$~-- is discrete. 
 Otherwise, we say that the relation is indiscrete. 
 It is known that there are wandering domains that contain both points with discrete and with
 indiscrete orbit relations~\cite{vassoetalorbitrelations}. Using the results of~\cite{REM24},
 one can also show that there exists a wandering domain $U_0$ for which the grand orbit of some $z_0$ under
 $f$ is dense in $U_0$. In the infinitesimally non-contracting case, the situation is much
 simpler. If such a wandering domain is essentially thick, then the orbit relation is discrete;
 if it is essentially thin, then the closure of any grand orbit in $U_0$ is locally
 contained in \lr{a leaf of} the unique contracting singular foliation. See Proposition~\ref{prop:GOrelations}.

\subsection*{Ideas of the proofs}
The most challenging part of our main theorem (Theorem~\ref{thm:main}) is proving 
Theorems~\ref{thm:EI} and~\ref{thm:semi-contracting}. The corresponding results for entire functions span many papers~\cites{KS08,BRS13,Fer22}, and rely on structural theorems available only to entire functions. Most crucially, in~\cite{BRS13}, the authors use properties of the maximum modulus to obtain a harmonic function $h$ on
$U_0$ that encodes (in a certain sense) the rate of growth of $z\in U_0$ relative to a base point $z_0\in U_0$; see~\eqref{eqn:BRSu}. The properties of $h$ lead to the existence of an absorbing sequence of large (round) annuli, and the contracting foliation from~\cite{Fer22}*{Theorem~1.2} is given by the level lines of $h$.

For meromorphic functions, we do not have the maximum modulus to work with, and orbits in a multiply connected wandering domain need not tend to infinity. Hence the above strategy cannot be 
applied, and new ideas are required to deal with the non-contracting thin case. To overcome this challenge,
we consider universal coverings of the domains $U_n$ (Lemma~\ref{lemma:lifts}), 
and show that the corresponding groups of 
deck transformations converge to 
a continuous one-dimensional limit group (Lemma~\ref{lem:homomorphisms}). The orbits of points under this limit group give rise to the desired contracting foliation. For large $n$, the point $f^n(z_0)$ is contained in an annulus of large modulus bounded by closed leaves in the contracting foliation. As we show in Section~\ref{sec:entire},
the foliation is again given by the level lines of a positive harmonic function~--
but here the function is obtained \emph{a posteriori} from the contracting foliation.

(Another approach to the existence of attracting foliations in multiply connected wandering domains of entire functions was outlined by Shishikura~\cite{shishikuratalk} in unpublished work. His argument uses the existence of collars (annuli of large modulus) around short geodesics in $U_n$. Using this theorem and the Schwarz--Pick lemma, Shishikura is able to show the existence of an absorbing sequence of such collars. He then constructs the desired foliation as a limit of the 
concentric foliations of the collars. 
This strategy can 
also be adapted to our setting. Instead, we construct the foliation directly; 
 the absorbing annuli 
and their key properties appear naturally via the structure of this
foliation.)

The remainder of Theorem~\ref{thm:main} is proved by applying or adapting known 
techniques. In order to show the existence of locally but not globally eventually 
isometric wandering domains (Theorem~\ref{thm:ex}), we first show that such an 
example can be realised by a \emph{non-autonomous} sequence of 
holomorphic covering maps between plane domains (see Definition~\ref{defn:nonautonomous}). Each of these covering maps is the restriction of a Blaschke product of the disc of degree two, chosen 
(inductively) to be
sufficiently close to the identity on a large subset of the disc (Proposition~\ref{prop:model}). We then
realise this model dynamics by a multiply connected wandering
domain of a meromorphic function, using a recent 
result of Evdoridou, Mart\'i-Pete, and the second author~\cite{REM24}.

\subsection*{Non-autonomous systems of maps between hyperbolic surfaces}
Although we have stated our results above for wandering domains of meromorphic functions, most proofs do not rely on the fact
that the functions between them are globally defined, and make
sense in the following more general setting.

\begin{defi}\label{defn:nonautonomous}
A \emph{non-autonomous system of holomorphic maps between hyperbolic Riemann surfaces}~-- or, more succinctly, a \emph{non-autonomous system on hyperbolic surfaces}~-- is a sequence $F = (U_n,f_n)_{n=0}^{\infty}$, where each $U_n$ is a hyperbolic Riemann surface and each $f_n\colon U_n\to U_{n+1}$ is non-constant and holomorphic.
\end{defi}
The concepts introduced for orbits of wandering domains above all extend verbatim to the case of sequences as above and all our results above remain valid and will be proved in this more general setting, with the exception of certain statements about eventual connectivities (compare Theorem~\ref{thm:main-nonautonomous}). In particular, these results extend immediately to wandering domains in the more general setting of Epstein's \emph{Ahlfors islands maps}; compare~\cite{exoticbakerdomains}.

\subsection*{Structure of the paper}
The structure of this paper is as follows. In Section \ref{sec:prelim}, we collect some concepts and results on approximation theory and hyperbolic geometry to make use of during our proofs. Section \ref{sec:uniform} is dedicated to proving Propositions \ref{prop:trich} and \ref{prop:injglobal}. Section~\ref{sec:global} is 
the heart of the paper: we consider essentially thin wandering domains, proving
Theorems \ref{thm:EI} and \ref{thm:semi-contracting}. 
In Section~\ref{sec:entire}, we compare the results of Section~\ref{sec:global} with
those for multiply connected wandering domains of entire functions (this section
is independent of the proofs of our main results). We deduce Theorem~\ref{thm:main} in Section~\ref{sec:mainproof}. Finally, in Section \ref{sec:ex}, we discuss how to realise all cases of Theorem~\ref{thm:main}. In particular, we prove Theorem \ref{thm:ex}. 

\subsection*{Basic notation}
As usual, $\DD$, $\Cx$ and $\Chat$ denote the unit disc, complex plane and
 Riemann sphere, respectively. Closures and boundaries are taken in
  $\Cx$, unless explicitly stated otherwise. The Euclidean disc of radius $\delta$ centred
 at $z_0$ is denoted $D(z_0,\delta)$. 
 
 We frequently work with the hyperbolic metric on a 
  hyperbolic domain $U\subset\Cx$; i.e., an open connected set such that
  $\#\Cx\setminus U>1$. We refer to~\cite{BM06} for background on 
  hyperbolic geometry, and to Section~\ref{sec:prelim} for  further
  discussion. 
   Hyperbolic distance and hyperbolic diameter in $U$
   are denoted by $d_U$ and $\diam_U$, respectively. The hyperbolic
   disc of radius $\delta>0$ is denoted 
     \[ D_U(z_0,\delta)\defeq \{z\in U\colon d_U(z,z_0)<\delta \}\lr{.} \]

\lr{\subsection*{Authors' statement on the use of generative AI}
The authors prepared the original manuscript of this paper 
without the use of generative AI. After submission for
publication, generative AI (ChatGPT Codex with models 6/6.2 Astra and 6.1 Sol) 
was used to formalise the results using the 
Lean proof assistant, and also to undertake an editorial review of the manuscript. It also
generated Figures~\ref{fig:onion} and~\ref{fig:prongs} according to the authors' specification. As part of the formalisation process, 
a separate AI agent (Anthropic Claude Opus 5.5) was used to review the 
fidelity of the formal Lean statements to their human-language
versions in the manuscript. 

During the formalisation process, the AI agent encountered a number of
mathematical gaps and inaccuracies. The AI was specifically prompted not to 
suggest any corrections of a mathematical nature. All identified gaps were resolved
explicitly through mathematical decisions by the authors; corrections to the manuscript
were carried out strictly by the authors. Hence all text in the
article was written by the authors, who take full responsibility for the 
final contents of the paper and any inaccuracies that remain. Every numbered statement
(theorem,
proposition, lemma, corollary) in the paper has a Lean-verified analogue in the 
finished formalisation~\cite{palomarMCWD}.%
}
\subsection*{Acknowledgements}
The first author thanks Phil Rippon and Gwyneth Stallard for introducing him to the internal dynamics of wandering domains.

\section{Preliminaries}\label{sec:prelim}
\subsection*{\lr{Prescribed wandering dynamics}}
The proof of Theorem \ref{thm:ex} will use the following result of Evdoridou, Martí-Pete, and Rempe \cite{REM24} for constructing meromorphic functions with prescribed behaviour.
\begin{thm}[Realising prescribed wandering dynamics] \label{thm:magic}
Let $(X_n)_{n\geq 0}\subset \Cx$ be a sequence of pairwise disjoint compact sets such that $X_n\to\infty$ as $n\to\infty$. Let $\varphi$ be \lr{an open} 
holomorphic function defined in
a neighbourhood of $X\defeq \bigcup_{n=0}^{\infty} X_n$ such that
$\varphi(X_n)\subset X_{n+1}$ and $\varphi(\partial X_n)\subset \partial X_{n+1}$ for all
$n\geq 0$. Let $(\varepsilon_n)_{n\geq 0}$ be a sequence of positive numbers. Then there exists a meromorphic function $f\colon\Cx\to\Chat$ and a $\mathcal{C}^{\infty}$ diffeomorphism
$\theta\colon \Cx\to\Cx$ such that 
\begin{enumerate}[(a)]
    \item $f\circ\theta = \theta\circ \varphi$ on $X$;
    \item $\theta(\partial X) \subset J(f)$;
    \item $\theta$ is conformal on $\intr(X)$; and
    \item $|\theta(z) - z| \leq \varepsilon_n$ for $z\in X_n$.
\end{enumerate}
\end{thm}

\subsection*{\lr{Singular foliations}}
 \lr{To define singular foliations, set $Q_2\defeq \{a+ib\colon \lvert a\rvert, \lvert b\rvert < 1\}$.
 The partition $\mathcal{Q}_2$ of $Q_2$ into horizontal
 leaves is called the \emph{regular} local model. The \emph{one-pronged} local model 
 $\mathcal{Q}_1$ is the image of $\mathcal{Q}_2$ under the map $P_2(z)\defeq z^2$; this is 
 well-defined because the horizontal foliation $\mathcal{Q}_2$ is invariant under $z\mapsto -z$. 
 The one-pronged model $\mathcal{Q}_1$ is hence a partition of $Q_1 = P_2(Q_2)$; one element of the partition is
 the half-open interval $[0,1)$. All other elements are open analytic arcs. 

\begin{figure}
\begin{center}
    \def\svgwidth{\textwidth}
    \input{pronged-local-models.pdf_tex}
\end{center}
\caption{\lr{Local models for $m=1,2,3,4$.}\label{fig:prongs}}
\end{figure}

 The \emph{$m$-pronged} local model $\mathcal{Q}_m$ is the pre-image of $\mathcal{Q}_1$ under 
 $P_m(z) \defeq z^m$. More precisely, 
   it is the partition of $Q_m\defeq P_m^{-1}(Q_1)$ into 
 path-connected components of $P_m$-preimages of leaves of $\mathcal{Q}_1$. Note that the $2$-pronged model is the same as the regular local model. (See Figure~\ref{fig:prongs}.)
 
 Now a \emph{singular foliation} of a Riemann surface $X$ is a partition $\mathcal{F}$ of $X$ 
  into path-connected subsets, called \emph{leaves}, such that the following holds.
  For every $z\in X$ there is a neighbourhood $U_z$ and a positive integer
  $m_z$ such that, in $U_z$, the foliation becomes homeomorphic to the $m_z$-pronged
  local model. That is, consider the partition of $U_z$ into ``local leaves''; i.e., 
  into path-connected components of $L\cap U_z$ for $L\in \mathcal{F}$. Then 
  there is a homeomorphism 
  $U_z\to Q_{m_z}$ that maps the local leaves in $U_z$
  to the leaves of $\mathcal{Q}_{m_z}$, with $z$ corresponding to $0$.\footnote{%
\lr{Often, a singular point of a foliation is instead considered to form a (degenerate) leaf, with a number
of other leaves ending at it. For us, it is more convenient to consider these as part of the same (singular) leaf.}} A point $z\in X$ with $m_z\neq 2$ is a \emph{singular} point of $\mathcal{F}$;
observe that the singular points of $\mathcal{F}$ are a closed and discrete subset of $X$.
  
 As defined in the introduction, a singular foliation is \emph{harmonic} if each $U_z$ can be
  chosen such that the local leaves in 
  $U_z$ are the path-connected components of the 
  level sets of some (necessarily non-constant) harmonic function on $U_z$. 
  Equivalently,  
  $m_z$ is even for all $z$ and the homeomorphism to a local model can
  be chosen to be a conformal isomorphism. 
  
  All foliations that occur in this paper are harmonic.
     (McMullen~\cite{McMullenFoliations} instead defines harmonic foliations using
     global harmonic $1$-forms, which is more restrictive than our definition and means
    that the resulting foliations can be considered as \emph{measured} foliations. We note that
    it follows from our proofs that the foliations
    in Theorems~\ref{thm:EI} and~\ref{thm:semi-contracting} are also harmonic in McMullen's sense.)}

\subsection*{\lr{Hyperbolic geometry}}
We now review some important results and concepts in hyperbolic geometry, beginning with the definition of hyperbolic distortion. 
\begin{defi}[Hyperbolic distortion]\label{def:distortion}
 Let $f\colon U\to V$ be a holomorphic map between hyperbolic Riemann surfaces, and let 
   $z_0\in U$. The \emph{hyperbolic distortion} of $f$ at $z_0$ is
     \begin{equation}
         \label{eqn:hypdistortion} 
       \| \Deriv \lr{f(z_0)}\|_U^V \defeq \lim_{z\to z_0} \frac{d_V(f(z),f(z_0))}{d_U(z,z_0)}.
       \end{equation}
\end{defi}
If $z$ and $w$ are local coordinates of $U$ and $V$ near $z_0$ and $f(z_0)$, 
we can express the hyperbolic metrics of $U$ and $V$ as
$\rho_U(z)\lr{\lvert \deriv z\rvert}$ and $\rho_V(w) \lr{\lvert \deriv w\rvert}$ in these coordinates, with
$\rho_U(z), \rho_V(w)>0$. In terms of these densities,  the hyperbolic distortion of $f$ at $z_0$ can be expressed as
\begin{equation}\label{eqn:hypdistortion_coords} 
    \|\Deriv f(\lr{z_0})\|_U^V = \frac{\rho_V\left(f(\lr{z_0})\right)|f'(\lr{z_0})|}{\rho_U(\lr{z_0})}\lr{.} \end{equation}
\lr{When} $f\colon U\to U$ is a holomorphic
self-map of $U$,
we also write $\|\Deriv f(\lr{z_0})\|_U$ instead of
$\|\Deriv f(\lr{z_0})\|_U^U$.

The \emph{Schwarz--Pick theorem}~\cite{BM06}*{Theorem~3.2} states that holomorphic maps do not increase the hyperbolic metric. 
\begin{thm}[Schwarz--Pick theorem]\label{thm:SP}
Let $f\colon U\to V$ be a holomorphic map between hyperbolic Riemann surfaces. Then
\begin{equation}\label{eqn:schwarzpick} \|\Deriv f(z)\|_U^V \leq 1 \text{ for all $z\in U$}, \end{equation}
with equality if and only if $f\colon U\to V$ is an unbranched covering map.
\end{thm}

Beardon~\cite{beardonschwarzpickderivatives} proved a version of the 
Schwarz--Pick theorem for the hyperbolic distortion; see also \cite{BM06}*{Theorem~11.2}.

\begin{thm}[Schwarz--Pick theorem for hyperbolic distortion]\label{thm:schwarzpickdistortion}
 Let $f\colon U\to V$ be a holomorphic map between hyperbolic Riemann surfaces but not a covering map. Then, for all $z,w\in U$, 
  \[ d_{\DD}\bigl(\|\Deriv f(z)\|_U^V ,\|\Deriv f(\lr{w})\|_U^V\bigr) \leq  2d_{U}(z,w). \]
\end{thm}

\subsection*{\lr{Non-autonomous systems and preferred lifts}}
As mentioned \lr{in the introduction}, we will prove \lr{most} of our results for the general case of non-autonomous systems
 \lr{as defined in Definition~\ref{defn:nonautonomous}}; the dynamics of meromorphic functions on an orbit of wandering domains is a special case. We use the following notation.

\begin{defi}[Non-autonomous iteration]
Let $F=(U_n,f_n)_{n=0}^{\infty}$ be a non-autonomous system on hyperbolic surfaces, and let $z_0\in U_0$. For $k,n\geq 0$, we define 
  \[ f_n^{k} \defeq f_{n+k-1}\dots f_n \colon U_n\to U_{n+k}. \]
The sequence $(z_n)_{n=0}^{\infty}$ defined by $z_n\defeq f_0^{n}(z_0)\in U_n$ is called the \emph{orbit} of $z_0$. As in~\eqref{eq:distortionseq}, we define the sequence of hyperbolic distortions along the orbit of $z_0$ by
\[ \lambda_n(z_0) \defeq \lambda_n(z_0,F) \defeq \|\Deriv f_n(z_n)\|_{U_n}^{U_{n+1}}.\]
\end{defi}
In the case of a meromorphic function $f$ with an orbit $(U_n)_{n=0}^{\infty}$ of 
wandering domains, we have $f_n^{k}= f^{k}|_{U_n}$, and $(z_n)_{n=0}^{\infty}$ is the
orbit of $z_0$ in the usual sense.

It is often useful to lift the system $F$ to the universal coverings of 
the surfaces $U_n$. (See \cite{Fer24}*{Lemmas 2.3 and 3.2}.)
\begin{lemma}[Preferred lifts]\label{lemma:lifts}
Let $F=(U_n,f_n)_{n=0}^{\infty}$ be a non-autonomous system on hyperbolic surfaces, and let
$z_0\in U_0$ with orbit $(z_n)_{n=0}^{\infty}$. 
  
There are sequences $(\pi_n)_{n=0}^{\infty}$ of universal covering maps
$\pi_n\colon \DD\to U_n$ and non-constant holomorphic functions 
$g_n\colon \DD\to\DD$ with the following properties for $n\geq 0$:
\begin{enumerate}[(i)]
 \item $\pi_n(0)=z_n$;\label{item:pi-image}
 \item $g_n(0)=0$;\label{item:gn-fixes-origin}
 \item $f_n\circ \pi_n = \pi_{n+1}\circ g_n$;\label{item:lift-relation}
 \item if $z_n$ is not a critical point of $f_n$, then $g_n'(0)>0$.\label{item:positive-derivative}
\end{enumerate}

We say that $G = (\DD,g_n)_{n=0}^{\infty}$, which is a non-autonomous system on $\DD$,
 is a \emph{preferred lift} of the
 system $F$ (via the universal coverings $(\pi_n)_{n=0}^{\infty}$). The system $G$ satisfies
 \begin{equation}\label{eqn:lift-distortion} g_n'(0) = \lambda_n(0,G) = \lambda_n(z_0,F) \end{equation}
 for all $n\geq 0$. 
\end{lemma}
\begin{proof}
The construction of the system $G$ proceeds recursively. Let $\pi_0\colon \DD\to U_0$ be any universal covering with $\pi_0(0)=z_0$.
   
Now suppose that $n\geq 0$, and that the maps $\pi_k$ have been defined for $k\leq n$ and that $g_k$ has been defined for $k<n$. Let $\tilde{\pi}\colon \DD\to \lr{U_{n+1}}$ be a universal covering with $\lr{\tilde{\pi}}(0)=\lr{z_{n+1}}$, and let $\tilde{g}\colon \DD\to\DD$ be the corresponding lift of $f_n$; that is, $\tilde{g}(0)=0$ and $f_n\circ \pi_n = \tilde{\pi}\circ \tilde{g}$. By the Schwarz--Pick theorem and the
   chain rule, we have
     \[ \lvert (\tilde{g})'(0)\rvert = \| \Deriv \tilde{g}(0)\|_{\DD} = 
      \| \Deriv f_n(z_n)\lr{\|}_{U_n}^{U_{n+1}}.\]
If $z_n$ is a critical point of $f_n$, set $g_n\defeq \tilde{g}$ and $\pi_{n+1}\defeq \tilde{\pi}$. Otherwise, postcompose $\tilde{g}$ with a rotation to obtain $g_n$ with $g_n'(0)>0$ and precompose $\tilde{\pi}$ with the inverse of this rotation to obtain $\pi_{n+1}$. 

This ensures properties~\ref{item:pi-image} to~\ref{item:positive-derivative} as well as~\eqref{eqn:lift-distortion} for $n$, and completes the recursive construction and the proof. 
\end{proof}

\subsection*{\lr{Geometric convergence of discrete groups}}
 \lr{In a certain sense, the notion of \emph{geometric convergence} generalises}
 the Hausdorff metric to closed subsets of topological spaces\lr{. For our purposes, it can be} 
 defined as follows\footnote{This is not the usual definition, but an equivalent characterisation; see \cite{BP92}*{Proposition E.1.2}.}.
\begin{defi}\label{def:geomconvergence}
Let $X$ be a locally compact metrisable topological space, and let $\mathcal{C}(X)$ denote the set of all closed subsets of $X$. We say that a sequence 
$(C_n)_{n=0}^{\infty}\subset\mathcal{C}(X)$ of subsets \textit{converges geometrically} to $C\in \mathcal{C}(X)$ if and only if \lr{both of the following hold}:
\begin{enumerate}[(i)]
    \item\label{item:limitelements} Any $x\in X$ for which there exists a subsequence $(n_k)_{k=0}^{\infty}$ and points $x_k\in C_{n_k}$ such that $x_k\to x$ as $k\to+\infty$ satisfies $x\in C$; and
    \item For any $x\in C$, there exist \lr{$n_0\geq 0$ and a sequence $(x_n)_{n=n_0}^{\infty}$} with $x_n\in C_n$ such that $x_n\to x$.\label{item:convpoints}
\end{enumerate} 
\end{defi}

We will apply this definition in the case where $X=\Aut(\DD)$, the Lie group of conformal automorphisms of $\DD$. The topology on $\Aut(\DD)$ is that of uniform convergence; recall that the map $\Aut(\DD)\to \DD\times S^1; \gamma\mapsto (\gamma(0), \arg \gamma'(0))$ is 
a homeomorphism. In order to effectively use the geometric topology to our advantage, we need the following:
\begin{thm}[Geometric topology for closed subgroups~\cite{BP92}*{Lemma E.1.4}]\label{thm:geomcompact}
The subset $\mathcal{S}$ of all closed subgroups of $\Aut(\DD)$ is compact in the geometric topology.
\end{thm}

\begin{thm}[Continuous limit groups]\label{thm:continuouslimits}
Suppose that $(\Gamma_n)_{n=0}^{\infty}$ is a sequence of torsion-free Fuchsian groups that converges to a group $\Gamma\in \mathcal{S}$. Furthermore suppose that there is a sequence $(\gamma_n)_{n=0}^{\infty}$ with $\gamma_n\in\Gamma_n\setminus\{\id\}$ such that $\gamma_n\to\id$ as $n\to\infty$.

Then $\Gamma$ is a continuous one-parameter subgroup of $\Aut(\DD)$. More precisely, $\Gamma$ is conjugate either to the group of real translations on the upper half-plane, or to the group of real translations on a horizontal bi-infinite strip.
\end{thm}
\begin{proof}
By \cite{MT98}*{Theorem 7.1}, the limiting group is Abelian. Since $\Gamma_n$ is torsion-free and discrete, \lr{different powers of $\gamma_n$ are distinct}. It follows easily that $\Gamma$ contains elements arbitrarily close to the identity and is therefore non-discrete. 
   
Two M\"obius transformations \lr{of infinite order} commute if and only if they have the same set of fixed points. Hence $\Gamma$ is either a one-parameter group of hyperbolic transformations, all sharing the same two fixed points on $\partial\DD$, or a one-parameter group of parabolic transformations, all sharing the same single fixed point of $\partial\DD$. By sending the two fixed points to the ends of a bi-infinite strip, or by sending the single fixed point to the boundary point of the upper half-plane at infinity, we obtain one of the two stated alternatives. 
\end{proof}

\subsection*{\lr{Injectivity radii}}
 \lr{To conclude the section,} we discuss injectivity radii of hyperbolic Riemann surfaces.
\begin{defi}[Injectivity radius]\label{def:inj}
Let $X$ be a hyperbolic Riemann surface, and let $p\in X$. The \textit{injectivity radius} of $X$ at $p$ is
\[ \inj_X(p) \defeq \inf\{\ell_X(\gamma)/2\colon \text{$\gamma\subset X$ is a non-trivial loop through $p$}\}. \]
If $X$ is simply connected, we define $\inj_X(p)$ to be infinite.
\end{defi}
An equivalent way of stating Definition \ref{def:inj} (and one that better explains the name) is that $\inj(p)$ is the supremum of $r > 0$ such that a universal covering $\varphi\colon\DD\to X$, normalised so that $\varphi(0) = p$, is injective in the hyperbolic ball $D_\DD(0, r)$ (cf.\ \cite{Thu97}*{Section 4.5}). We will use the following closely related characterisation. 

\begin{lemma}[Injectivity radii in terms of universal coverings]\label{lem:injformula}
Let $X$ be a \lr{multiply connected}
hyperbolic Riemann surface, let $p\in X$, and let $\pi\colon\DD\to X$ be a universal covering map such that $\pi(0) = p$. Let $\tilde p\in \pi^{-1}(\lr{p})\setminus\{0\}$ 
 be a point such that
\[ d\defeq d_{\DD}(0, \tilde p) = \min\{d_{\DD}(0,\tilde p)\colon \tilde{p}\in\DD\setminus\{0\}, \pi(\tilde{p}) = p\}\lr{.} \]
Then $\inj_X(p) = d/2$.
\end{lemma}
\begin{proof}
Join $0$ to $\tilde p$ by a geodesic arc $\gamma\lr{\subset}\DD$; then $\pi(\gamma)$ is 
a non-trivial closed curve through $p$ in \lr{$X$} (recall that any trivial loop in \lr{$X$} lifts to
a  trivial loop in $\DD$ under $\pi$). So, by definition, 
\[ \inj_X(p) \leq \frac{1}{2}\ell_X(\pi(\gamma)) = \frac{1}{2}\ell_\DD(\gamma) = \frac{d}{2}. \]
For the reverse inequality, let $\gamma\lr{\subset} X$ be some non-contractible closed curve passing through $p$. By the unique path-lifting property of $\pi$ \cite{Bre93}*{Theorem 3.3}, there
 is a lift $\tilde\gamma\subset \DD$ of $\gamma$ starting at $0$ and terminating at some 
 $\hat{p}\in\pi^{-1}(p)$ distinct from zero. We have 
\[ d \leq d_\DD(0, \hat{p}) \leq \ell_\DD(\tilde\gamma) = \ell_X(\gamma); \]
thus,
 \[ \inj_X(p) = \frac{1}{2}\inf_{\gamma} \ell_X(\gamma) \geq \frac{d}{2}.\qedhere \]
\end{proof}

 Definition~\ref{def:thickthin} generalises directly to a non-autonomous system $F=(U_n,f_n)_{n=0}^{\infty}$ as in Definition~\ref{defn:nonautonomous}, by setting
  \begin{equation} \label{eqn:deltan_nonaut} \delta_n(z) = \inj_{U_n}(f_0^n(z)) \end{equation} 
  for $n\geq 0$ and $z\in U_0$.

\section{Uniformity at the infinitesimal level}\label{sec:uniform}
In this section, we prove Propositions \ref{prop:trich} and \ref{prop:injglobal}, restated below in more general forms as Propositions~\ref{prop:trich_nonaut} and~\ref{prop:injglobal_nonaut}. In the proof of
Proposition \ref{prop:trich_nonaut}, we use the following formula for the hyperbolic distance between points in the unit disc (see \cite{BM06}*{p. 13}). 
\begin{lemma}[Hyperbolic distances]\label{lem:distD}
Let $z, w\in\DD$ be such that $0 \leq z \leq w < 1$. Then
\[ d_\DD(z, w) = \log\left(\frac{1 + w}{1 - w}\cdot\frac{1 - z}{1 + z}\right). \]
\end{lemma}
Using Lemma \ref{lem:distD}, we obtain the following bounds for the hyperbolic distortion:

\begin{lemma}[Hyperbolic distortion and hyperbolic distance]\label{lem:boundHD}
Let $g\colon X\to Y$ be a holomorphic map between hyperbolic Riemann surfaces. For $p\in X$, let $\lambda(p) \defeq \|\Deriv g(p)\|_X^Y$. Then, for any $z, w\in X$,
\[ e^{-2d_X(z, w)}(1 - \lambda(w)) \leq 1 - \lambda(z) \leq e^{2d_X(z, w)}(1 - \lambda(w)). \]
\end{lemma}
\begin{proof}
First, if $\lambda(z) = 1$, then by the Schwarz--Pick lemma $g$ is a covering map, and so $\lambda(w) = 1$ for every $w$ in $X$. Thus, in this case, the inequality is trivial.

Assume, for now, that $\lambda(z) \leq \lambda(w) < 1$. Applying Lemma \ref{lem:distD} gives us
\[ \frac{1 + \lambda(w)}{1 - \lambda(w)}\cdot\frac{1 - \lambda(z)}{1 + \lambda(z)} = e^{d_\DD(\lambda(z), \lambda(w))}. \]
Therefore,
\[ 1 - \lambda(z) = e^{d_\DD(\lambda(z), \lambda(w))}(1 - \lambda(w))\frac{1 + \lambda(z)}{1 + \lambda(w)} \leq e^{d_\DD(\lambda(z), \lambda(w))}(1 - \lambda(w)), \]
where the inequality follows from the fact that $0 \leq \lambda(z) \leq \lambda(w) < 1$.

By Beardon's Schwarz--Pick theorem for hyperbolic distortions (Theorem~\ref{thm:schwarzpickdistortion}),
\[ 1 - \lambda(z) \leq e^{2d_X(z, w)}(1 - \lambda(w)). \]
Reversing the roles of $z$ and $w$ in the proof gives the lower bound, and the proof is complete.
\end{proof}

We now prove Proposition \ref{prop:trich}, in the following more general form
for the setting of non-autonomous systems.

\begin{prop}\label{prop:trich_nonaut}
Let $F=(U_n, f_n)_{n=0}^{\infty}$ be a non-autonomous system on hyperbolic
 surfaces. Then, for all $n\geq 0$ and any $z, w\in U_0$,
\[ e^{-2d_{U_0}(z,w)} (1 - \lambda_n(w)) \leq 1 - \lambda_n(z) \leq 
e^{2d_{U_0}(z,w)}(1 - \lambda_n(w)). \]
In particular, whether the series
\[ \sum_{n=0}^{\infty} (1 - \lambda_n(z)) \]
converges, diverges, or has a finite number of non-zero terms does not depend on the choice of $z\in U_0$.
\end{prop}
\begin{proof}Let $n\geq 0$ and $z, w\in U_0$. Write $z_n \defeq f_0^n(z)$ and
$w_n\defeq f_0^n(w)$. By Lemma \ref{lem:boundHD},
\[ e^{-2d_{U_{n}}(z_n,w_n)}(1 - \lambda_n(w)) \leq 1 - \lambda_n(z) \leq e^{2d_{U_{n}}(z_n,w_n)}(1 - \lambda_n(w)), \]
while by Theorem~\ref{thm:SP} $d_{U_{n}}(z_n,w_n) \leq d_{U_0}(z, w)$.
\end{proof}

Next, we prove Proposition \ref{prop:injglobal}, again restated here in more
general form. 
\begin{prop}\label{prop:injglobal_nonaut}
Let $F=(U_n, f_n)_{n=0}^{\infty}$ be a non-autonomous system on hyperbolic
 surfaces. Let $z\in U_0$. 
If $(n_k)_{k=0}^{\infty}$ is a sequence such that $\delta_{n_k}(z)\to 0$, then $\delta_{n_k}(w)\to 0$ for all $w\in U_0$. 
(Recall from~\eqref{eqn:deltan_nonaut} that $\delta_n(w) \defeq \inj_{U_n}(f_0^n(w))$ for
$n\geq 0$ and $w\in U_0$.)
\end{prop}
\begin{proof}
\lr{We may assume that $\lr{\delta}_{n_k}(z)<\infty$ for all $k$.}
Let $(\DD,g_n)_{n=0}^{\infty}$ be a preferred lift of the system $F$ via universal 
covering maps $\pi_n\colon \DD\to U_n$.
For \lr{every $n\geq 0$}, set 
$z_n\defeq f_0^n(z)$\lr{. If $U_n$ is multiply connected, also} let
 $\tilde z_n\in\DD$ be a point of
 $\pi_n^{-1}(\lr{z_n})\lr{\setminus\{0\}}$ closest to the origin, i.e.\ so that
\[ |\tilde z_n| = \inf\{\lvert w\rvert\colon w\in\DD\setminus\{0\}, \pi_n(w) = \lr{z_n}\}. \]
For each \lr{$\tilde{z}_n$}, we find a deck transformation $\gamma_n\in\Aut(\DD)$ so that $\lr{\gamma_n}(0) = \tilde z_n$. By Lemma~\ref{lem:injformula}, the hypothesis $\inj_{U_{n_k}}(z_{n_k})\to 0$ implies that $\tilde z_{n_k}\to 0$ as $k\to+\infty$, and we claim that this implies that the sequence $(\gamma_{n_k})_{k\in\mathbb{N}}$ converges locally uniformly to the identity as $k\to+\infty$. First, since the \lr{automorphisms} of $\DD$ can be parametrised by their image and derivative at zero, we see (since $\tilde z_{n_k} = \gamma_{n_k}(0)\to 0$) that the sequence $(\gamma_{n_k})_{k=0}^{\infty}$ is precompact in $\Aut(\DD)$, and that any of its limit functions is an automorphism of $\DD$ fixing the origin. Since no $\gamma_{n_k}$ has fixed points in $\DD$, it follows from Hurwitz's theorem that the only possible limit function is the identity. Thus $\gamma_{n_k}\to \id$ as $k\to+\infty$, uniformly in the Euclidean metric and locally uniformly in the hyperbolic metric of $\DD$.

Now let $w\in U_0$ and set $w_n\defeq f_0^n(w)$
  for $n\geq 0$. Then, for each $n$,
  there exists $\lr{\widetilde{w}_n}\in\lr{\pi_n}^{-1}(w_n)$ such that
\[ d_\DD(\lr{\widetilde{w}_n,0}) = d_{U_n}(w_n, z_n) \leq d_{U_0}(w, z), \]
where the inequality comes from the Schwarz--Pick lemma. By the definition of the injectivity radius,
\[ \inj_{U_{n_k}}(w_{n_k}) \leq \frac{1}{2}d_\DD\left(\tilde w_{n_k}, \gamma_{n_k}(\tilde w_{n_k})\right). \]
The conclusion now follows from the fact that $\gamma_{n_k}\to \id$ as $k\to+\infty$.
\end{proof}

\section{Essential thinness, large annuli, and contracting foliations}\label{sec:global}
Throughout this and the next section, we fix a non-autonomous system \[F=(U_n,f_n)_{n=0}^{\infty}\] of holomorphic maps $f_n\colon U_n\to U_{n+1}$ between hyperbolic Riemann surfaces (in the sense of Definition~\ref{defn:nonautonomous}). We also fix an orbit $(z_n)_{n=0}^{\infty}$ and a corresponding preferred lift 
$G=(\DD,g_n)_{n=0}^{\infty}$ of $F$ as in Lemma~\ref{lemma:lifts}, via universal 
coverings $(\pi_n)_{n=0}^{\infty}$. For each $n\geq 0$, let $\Gamma_n\leq \Aut(\DD)$ be the group of
 deck transformations of $\pi_n$. We begin with a general observation.
 
\begin{lemma}[Limit functions]\label{lem:limitfunctions}
  For every $k\geq 0$, as $n\to \infty$ the functions 
    $g_k^n\colon \DD\to \DD$ converge locally uniformly to a holomorphic function
    $G_k\colon \DD\to\DD$ with $G_k(0)=0$.
    Furthermore, $G_k = G_{k+1}\circ g_k$ for all $k\geq 0$.
    
    The systems $F$ and $G$ are 
    infinitesimally eventually isometric if and only if $G_k=\id$
    for all sufficiently large $k$. They are infinitesimally semi-contracting if and only if $G_k\to\id$
     locally uniformly but $G_k\neq \id$ for all $k$. The systems are 
     infinitesimally 
     contracting if and only if $G_k(z)=0$ for all $k\geq 0$ and all $z\in\DD$;
     in this case, they are globally 
     contracting.
 \end{lemma}
 \begin{proof}
The family $\mathcal{G}$ of all holomorphic self-maps of $\DD$ with $g(0)=0$ is compact with respect to the topology of locally uniform convergence. Indeed, the family is normal by (the weak version of) Montel's theorem~\cite{Ahl79}*{Theorem 5.15}, and any limit function is again holomorphic by Weierstrass's theorem~\cite{Ahl79}*{Theorem 5.1}. 

  Let $\mathcal{G}_k\subset\mathcal{G}$ denote the set of all limit functions of 
    $g_k^n$ as $n\to\infty$. Then, an element $G\in\mathcal{G}$ belongs to
    $\mathcal{G}_{k+1}$ if and only if 
    $G \circ g_k\in \mathcal{G}_k$. By the Schwarz--Pick theorem,
    for every $z\in\DD$
    the hyperbolic distortion $\|\Deriv g_k^n(z)\|_{\DD}$ 
    is non-increasing
    as $n\to\infty$. If these distortions converge locally uniformly to $0$ as 
    $n\to\infty$, then every 
    $\mathcal{G}_k$ contains only the constant function $z\mapsto 0$, and 
    the systems $F$ and $G$ are 
    infinitesimally and globally contracting. 

 Otherwise, let $K\subset \DD$ be a closed disc centred at $0$. Then     \begin{equation}\label{eqn:derivatives_tend_to_1} \lr{\inf}\{ \|\Deriv G(z)\|_{\DD}\colon G\in \mathcal{G}_k, z\in K\} \to 1 \end{equation}
   as $k\to\infty$. Suppose that 
   $(G_k)_{k=0}^{\infty}$
   is a sequence with $G_k\in \mathcal{G}_k$ for $k\geq 0$. Then it follows from~\eqref{eqn:derivatives_tend_to_1}
   that any limit function $\phi$ of 
   $G_k$ as $k\to\infty$ is a conformal isomorphism with $\phi(0)=0$;
   since $G_k'(0)\in [0,\infty)$, we must
   have $\phi = \id$. 
   
   So, for any $r>0$ and all
     sufficiently large \lr{$k$}, 
     all functions in \lr{$\mathcal{G}_k$} 
     are close to the identity
     on the hyperbolic disc $D_{\DD}(0,r)$. 
     Let \lr{$k\geq0$}. Any two limit functions $G,H\in \mathcal{G}_k$ 
     can be written as \lr{$G=G_n\circ g_k^{n-k}$
     and $H=H_n\circ g_k^{n-k}$} for all
     $n\geq k$ and some
     $G_n,H_n\in \mathcal{G}_n$. 
     Since $G_n\to \id$ and $H_n\to \id$
     as $n\to\infty$,
     we conclude that $G=H$. 
     (Recall that $g_k^n(D_{\DD}(0,r))\subset
     D_{\DD}(0,r)$.)
     
  We have proved that \lr{$\mathcal{G}_k$} contains
   a single element \lr{$G_k$}, and hence
   $g_k^n\to G_k$ locally uniformly as
   $n\to\infty$. We have also shown that
   $G_k\to\id$ locally uniformly.
   Recall that $F$ and $G$
     are infinitesimally 
     eventually isometric if and only if
     there is $k_0$ such that
     $g_k = \id$ for all $k\geq k_0$,
     and hence $G_k = \id$ for $k\geq k_0$. 
     Conversely, if $g_{k_0}\neq \id$ for
     some $k_0$, then $G_k\neq \id$ for $k\leq k_0$, since 
        \[ \|\Deriv G_k(z)\|_{\DD}\leq 
     \|\Deriv g_{k_0}(g_k^{k_0-k}(z))\|_{\DD} < 1\]
     for all $z\in\DD$. Hence, if
     $g_{k}\neq \id$ for infinitely many
     $k$, then $G_k\neq \id$ for all $k\geq 0$. This completes the proof.
 \end{proof}

We note the following consequence of the preceding lemma\lr{.}%
\begin{cor}[Infinitesimally contracting systems are globally contracting]\label{cor:contracting}
\lr{An infinitesimally contracting \lr{non-autonomous} system $F$ of holomorphic maps between hyperbolic Riemann surfaces is globally contracting.}
 \end{cor}
\lr{\begin{rmk}
 Observe that 
   this also follows by \lr{applying}~\cite{BEFRS19}*{Theorem 2.1} to a preferred lift $G$
   of $F$. 
   
 The converse also holds. For wandering domains of meromorphic functions, 
   this is \cite{Fer24}*{Theorem 1.1(a)}.
   In the general case, it follows from Theorem~\ref{thm:main-nonautonomous} below. 
\end{rmk}}

We now restrict to the case where $F$ is non-contracting, so that the limit
 functions $G_k$ are non-constant, and $G_k\to\id$ as $k\to\infty$. 
  
\begin{lemma}[Limit group]\label{lem:homomorphisms}
Suppose that the system $F$ is infinitesimally non-contracting. Then the groups $\Gamma_n$ converge geometrically to a torsion-free group $\Gamma\in \mathcal{S}$. Moreover, for every $k\geq 0$, there are group homomorphisms $\psi_k\colon \Gamma_k\to \Gamma_{k+1}$ and $\Psi_k\colon \Gamma_k\to \Gamma$  such that $g_k\circ \gamma = \psi_k(\gamma)\circ g_k$ and $G_k\circ \gamma = \Psi_k(\gamma)\circ G_k$ for all $\gamma\in\Gamma_k$. 

Either $\Gamma=\{\id\}$ and 
   $\delta_n(z_0)=\inj_{U_n}(z_n)\to \infty$,
   or the sequence 
   $\delta_n(z_0)$ is eventually finite
   and non-increasing. The group $\Gamma$ is discrete if and only if $F$ is
   essentially thick (i.e., if and only
   if $\delta_n(z_0) \not\to 0$). 
    Otherwise, $\Gamma$ is a continuous one-parameter
     group as in Theorem~\ref{thm:continuouslimits}. 
     
  If $F$ is infinitesimally eventually isometric, then $\psi_k=\Psi_k=\id$ for all sufficiently large $k$.
 \end{lemma}
 \begin{proof}
The existence of the group homomorphisms $\psi_k$ is a consequence of the basic properties of covering maps; see~\cite{Aba23}*{Proposition 1.6.14}. For $k\geq 0$ and $n\geq 0$, define $\psi_k^{n}\defeq \psi_{k+n-1}\circ \dots \circ \psi_k\colon \Gamma_k \to \Gamma_{k+n}$. Then $g_k^n\circ \gamma = \psi_k^n(\gamma)\circ g_k^n$. Since $G$ is infinitesimally non-contracting, we have $\lr{G_k}\to \id$ as $k\to\infty$. Since $(g_k^n)'(0) \geq \lr{(G_k)'}(0)  \to 1$ for all $n\geq k$, we have $g_k^n \to \id$ uniformly in $n\geq k$ as $k\to\infty$. 
We conclude that $\psi_k^n(\gamma)$ is uniformly close to $\gamma$ when $k$ is large, independently of $n$. It follows that the sequence $(\psi_k^n(\gamma))_{n=0}^{\infty}$ is Cauchy for every $k$, and hence converges to an element $\Psi_k(\gamma)$ of $\Aut(\DD)$ as $n\to\infty$. Moreover, this convergence is uniform in the following sense: If $K\subset \Aut(\DD)$ is compact, then $\Psi_k(\gamma)$ is uniformly close to $\gamma$ for all sufficiently large $k$ and all $\gamma\in K\cap \Gamma_k$.

Observe that $\Psi_k(\Gamma_k) = \Psi_{k+1}(\psi_k(\Gamma_k))\subset \Psi_{k+1}(\Gamma_{k+1})$ for all $k\geq 0$. Let $\Gamma$ be the closure of the increasing union of all $\Psi_k(\Gamma_k)$. We claim that $\Gamma_k\to \Gamma$. Indeed, let $\gamma\in\Gamma$. Then
   there is a sequence $(\gamma_k)_{k=0}^{\infty}$ such that
   $\Psi_k(\gamma_k)\to \gamma$.
   Moreover, each $\Psi_k(\gamma_k)$ is the limit of
   $\psi_k^j(\gamma_k)\in \Gamma_{k+j}$
   as $j\to\infty$. 
   It follows that 
   $\psi_{k(n)}^{n-k(n)}(\gamma_{k(n)})
   \to \gamma$ as $n\to\infty$ if
   $k(n)\to\infty$ sufficiently slowly. 
   This establishes condition~\ref{item:convpoints}
   of Definition~\ref{def:geomconvergence}. 
   On the other hand,
    suppose that a sequence $(\gamma_{n_k})_{k=0}^{\infty}$ with
    $\gamma_{n_k}\in \Gamma_{n_k}$ 
    is given such that
   $\gamma_{n_k}\to \gamma\in\Aut(\DD)$ as $k\to\infty$. Then
   $K = \{\gamma_{n_k}\colon k\geq 0\}
   \cup \{\gamma\}\subset \Aut(\DD)$ is 
   compact. By the above
   uniformity statement, we have 
   $\Psi_{n_k}(\gamma_{n_k})\to \gamma$,
   and hence $\gamma\in\Gamma$. This
   establishes~\ref{item:limitelements}
   of Definition~\ref{def:geomconvergence}.
   Hence $\Gamma_k\to \Gamma$, as claimed.

 Next, \lr{if $\Gamma_k\neq\{\id\}$, then}  
  \begin{align*} 2\delta_k(z_0) = 
    2\inj_{U_k}(z_k) &= 
    \min\{d_{\DD}(0,z) \colon z\in\DD^*, \pi_k(z)=z_k\} \\ &= 
    \min\{ d_{\DD}(0,\gamma(0))\colon 
      \gamma\in \Gamma_k\setminus\{\id\} \}\end{align*}
 by Lemma~\ref{lem:injformula}\lr{, and $\inj_{U_k}(z_k)=\infty$ otherwise}. 
 Let $R>0$ and let
 $K$ be the closed hyperbolic disc
 of radius $R$ centred at $0$. There is
 $k_0$ such that $g_k$ is injective on $K$
 for all $k\geq k_0$. Suppose that 
 $2\delta_k(z_0)<R$ for some $k\geq k_0$, 
 and let $\gamma\in\Gamma_k\setminus\{\id\}$
 such that $d_{\DD}(0,\gamma(0))\lr{=2\delta_k(\lr{z_0})}$.
 Then 
  \[ \psi_k(\gamma)(0) =
      \psi_k(\gamma)(g_k(0)) =
      g_k(\gamma(0))\lr{\neq 0}.\]
 By the Schwarz--Pick theorem,
  \[ 2\delta_{k+1}(z_0) \leq  
      d_{\DD}(0,g_k(\gamma(0))) 
      \leq d_{\DD}(0,\gamma(0)) =
      2\delta_k(z_0). \]
 We conclude that either $\delta_k(z_0)\to\infty$, in which
  case  $\Gamma = \{\id\}$,
  or the sequence $\delta_k(z_0)$ is
  eventually non-increasing. 

 In particular, $F$ is essentially thin if and only if $\delta_k(z_0)\to 0$. In this case, $\Gamma$ is \lr{a} continuous one-parameter subgroup of $\Aut(\DD)$ by Theorem~\ref{thm:continuouslimits}.

Otherwise, $F$ is essentially thick, $\id$ is an isolated element of $\Gamma$, and hence $\Gamma$ is discrete.

If $F$ is infinitesimally eventually isometric, then there is $k_0$ such that $f_k$ is a covering map, and hence $g_k = \id$, for all $k\geq k_0$. It follows that $\psi_k=\Psi_k=\id$ for $k\geq k_0$.
 \end{proof}

\begin{cor}\label{cor:doublyconn}
If the system $F$ is infinitesimally eventually isometric and essentially thin, then, for all sufficiently large $n$, $U_n$ is doubly connected and $f_n\colon U_n\to U_{n+1}$ is a finite-degree covering map.
\end{cor}
\begin{proof}
    For sufficiently large $n$, $\Psi_n=\id$, and hence $\Gamma_n$ is a nontrivial discrete subgroup of the limit group $\Gamma$, which is isomorphic to $\RR$. Hence $\Gamma_n$ is isomorphic to $\mathbb{Z}$, and $U_n$ is doubly connected. 

That $f_n$ is a covering map is a simple consequence of the Schwarz--Pick theorem. Its degree is given by the index of $(f_n)_*(\pi_1(U_n))\simeq \mathbb{Z}$ as a subgroup of $\pi_1(U_{n+1})\simeq \mathbb{Z}$, which must be finite.
\end{proof}

Let us now assume that the system $F$ is essentially thin, so that the limit group $\Gamma$ is a continuous one-parameter group, conjugate to the group of translational isometries of either a finite horizontal strip (the \emph{hyperbolic} case) or of the upper half-plane (the \emph{parabolic} case). Let $\limitF$ be the  foliation of $\DD$ by the orbits of $\Gamma$. In the parabolic case $\limitF$ consists of the horocycles based at the unique common fixed point of the elements of $\Gamma$, while in the hyperbolic case it consists of arcs of circles connecting the two common fixed points of the elements of $\Gamma$. Also, let $\mathcal{L}$ be the foliation of $\DD$ perpendicular to $\limitF$. In the parabolic case, its leaves are the geodesics connecting the fixed point of $\Gamma$ to other points of $\partial\DD$. In the hyperbolic case, the leaves are the geodesics of $\DD$ perpendicular to the unique geodesic connecting the two fixed points of $\Gamma$.

\begin{thm}[\lr{Foliations for essentially thin non-contracting systems}]\label{thm:foliations}
For every $k\geq 0$, the pull-back  $\discF_k$ of $\limitF$ under $G_k$ is \lr{a harmonic}  singular foliation that is invariant under $\Gamma_k$. Hence the image of $\discF_k$ under  $\pi_k$ is 
\lr{a harmonic} singular foliation $\downF_k$ of $U_k$, and $f_k$ maps 
 \lr{$\downF_k$ to $\downF_{k+1}$}. 

 Similarly, the pull-back $\discL_k$
  of $\mathcal{L}$ is \lr{a harmonic} 
  singular foliation of $\DD$, and its
  image $\downL_k$ under
  $\pi_k$ is \lr{a harmonic}  singular foliation
  of $U_k$.

  The foliations 
  $\downF_k$ 
  and $\downL_k$ depend
  only on the system $F$, not on the choice
  of $z_0$ or $G$. 
\end{thm}
\begin{proof}
  The pull-back of any \lr{harmonic} singular
  foliation under an \lr{open} analytic function
  is again \lr{a harmonic} singular foliation.

  The foliations $\limitF$ and $\mathcal{L}$ are invariant under
  $\Gamma$. Hence, if $\gamma\in\Gamma_k$,
  then $\Psi_k(\gamma)$ maps leaves 
  of either foliation to leaves, 
  and it follows
  that $\gamma$ maps leaves of 
  $\discF_k$ and $\discL_k$ to
  leaves. This property
  ensures that 
  $\downF_k$ and $\downL_k$ are well-defined
  singular foliations on $U_k$. 

That the foliations of $U_k$ depend only on the system $F$ follows easily from the construction, since different choices of $G$ and $z_0$ would yield foliations $\discF_k'$ and $\discL_k'$ conjugate to $\discF_k$ and $\discL_k$ respectively.
\end{proof}

\begin{thm}[\lr{Dynamical properties of the foliations}]\label{thm:dynfoliations}
Let $z,w\in U_0$ be such that $z_n\defeq f_0^n(z)\neq f_0^n(w)\eqdef w_n$ for all $n\geq 0$. Set $d_n(z,w)\defeq d_{U_n}(z_n,w_n)\neq 0$.

Then $\lim_{n\to\infty} d_n(z,w) = 0$ if and only if $z_n$ and $w_n$ belong to the same leaf of $\downF_{n}$ for sufficiently large $n$. 

If $F$ and $G$ are infinitesimally eventually isometric, the sequence $(d_n(z,w))_{n=0}^{\infty}$ is eventually constant if and only if $z_n$ and $w_n$ belong to the same leaf of \lr{$\downL_n$} for sufficiently large $n$.

  In all other cases,
   \lr{$d_n(z,w)$} \lr{converges nontrivially}  
    to a positive limit as $n\to\infty$. 
\end{thm}
\begin{proof}
   \lr{Without loss of generality, we may assume that the preferred lift $G$ is chosen with respect to
   the orbit of $z$; i.e., $z=\pi_0(0)$.}
   Let \lr{$\omega_0$ be a preimage of $w$ under
    $\pi_0$ and set} 
    $\omega_n\defeq g_0^n(\lr{\omega_0})$.
    Then 
    $d_n(z,w)=d_{\DD}(\lr{0},
    \Gamma_n(\omega_n))$. Recall
    that the sequence $d_n(z,w)$ is
    non-increasing by the Schwarz--Pick
    theorem. 

 \lr{Let $K_n$ be the leaf of $\downF_n$ containing $w_n=\pi_n(\omega_n)$,
   let $\tilde{K}_n$ be the leaf of $\discF_n$ containing $\omega_n$, and let
   $K$ be the leaf of $\limitF$ containing $\omega\defeq G_0(\omega_0)=G_n(\omega_n)$.
   Then $K_n = \pi_n(\tilde{K}_n)$ and $G_n(\tilde{K}_n)\subset K$. We 
   claim that
    \begin{equation}\label{eqn:leafdistances} d_n(z,w) \geq \dist_{U_n}(z_n,K_n) = \dist_{\DD}(0,\tilde{K}_n) 
        \geq \dist_{\DD}(0,K)\end{equation}
    for sufficiently large $n$.  
    The first inequality holds by definition and the
    final inequality follows from the 
    Schwarz--Pick theorem.
    To prove the equality, recall that $\Gamma_n\to \Gamma$ and $G_n\to\id$. 
    It follows that, if $0<r<1$ and $n$ is sufficiently large,
    \begin{equation}\label{eqn:localleavesagree}
    D(0,r)\cap \tilde{K}_n = D(0,r)\cap G_n^{-1}(K).\end{equation} 
    Since $\pi_n^{-1}(K_n)\subset G_n^{-1}(K)$, the claim follows.
    }

\lr{Next, we claim that the first three quantities in~\eqref{eqn:leafdistances}
 converge to the final one as $n\to\infty$. Indeed,
  \[ d_n(z,w) = \min_{\gamma\in\Gamma_n} \dist_{\DD}(0,\gamma(\omega_n)).\]}
  As $n\to\infty$, 
  \lr{$\Gamma_n(\omega_n)$ converges to $\Gamma(\omega)=K$.} 

\lr{In particular, $d_n(z,w)\to 0$ if and only if $0\in K$. 
  By~\eqref{eqn:localleavesagree}, 
  $0\in K$ if and only if $0\in \tilde{K}_n$ for sufficiently large $n$, and by~\eqref{eqn:leafdistances},
  this is equivalent to $z_n\in K_n$, as claimed.} 

\lr{Now suppose that $0\notin K$. Then $d_n(z,w)$ is eventually constant 
    if and only if equality holds in both inequalities of~\eqref{eqn:leafdistances}
    for sufficiently large $n$. By the Schwarz--Pick theorem, the final
    two terms can agree only if 
    $G_n$ is a conformal isomorphism; i.e., $G_n=\id$ and $F$ is infinitesimally
    eventually isometric. So let us assume that this holds.}

 \lr{The point of $K$ that minimises the distance to the origin
 is the unique intersection point
  $\tilde{\omega}$
  of $K$ with the leaf $L$ of $\limitL=\discL_n$ containing $0$. 
   (This is a consequence of the definition of $\limitF$ and $\limitL$.)
 Thus $d_n(z,w)=\dist_{\DD}(0,K)$ means that 
  there is $\gamma\in\Gamma_n\subset\Gamma$
  such that $\gamma(\omega)=\tilde{\omega}$.}

 \lr{The latter implies that
  $w_n = \pi_n(\omega) = \pi_n(\tilde{\omega})$ belongs to the
  leaf $\pi_n(L)$ of $\downL_n$ containing $z_n$. Conversely, if
  $w_n\in \pi_n(L)$, then $\omega_n\in \Gamma_n(L)\cap K = \Gamma_n(\tilde{\omega})$, 
  implying equality as noted above.}
\end{proof}
\lr{\begin{rmk}
    The preceding theorem implies the following characterisation of
     $\downF_0$. Let $z\in U_0$; then the leaf of $\downF_0$ through $z$ is the
     connected component of $\{w\in U_0\colon d_n(z,w)\to 0\}$ containing $z$. This shows that
     the contracting foliation $\downF_0$ is unique. An analogous argument
     applies to the eventually isometric foliation in the infinitesimally
     eventually isometric case: The leaf of $\downL_0$ through $z\in U_0$ is the
     connected component of $z$ in $\{w\in U_0\colon d_n(z,w)\text{ is eventually constant}\}$. 
\end{rmk}}

The existence of large annuli in $U_n$, for sufficiently large $n$, is a consequence of the following result. 

\begin{thm}[Annuli] \label{thm:annuli}
Let $K\subset U_0$ be compact and connected\lr{, and} not
contained in a leaf of $\downF_0$. Then there is $n_0=n_0(K)\geq 0$ such that,
for every $n\geq n_0$, all leaves of $\downF_n$ intersecting $f_0^n(K)$ are 
\lr{homotopically non-trivial} simple closed curves. The union $A_n(K)\subset U_n$ of all leaves of $\downF_n$ intersecting $f_0^n(K)$ forms a 
\lr{homotopically non-trivial} closed annulus with $\Mod A_n(K)\to\infty$ as $n\to\infty$, and $f_n\colon A_n(K)\to A_{n+1}(K)$ is a covering map for every $n\geq n_0$. 
\end{thm}
\begin{proof} 
Let $K\subset U_0$ be as in the statement, and choose $R > \diam_{U_0}(K \cup \{\pi_0(0)\})> 0$. Then $K\subset \pi_0(D_R)$, where $D_R \defeq D_\DD(0, R)$. By the Schwarz--Pick theorem, $f_0^n(K)\subset \pi_n(D_R)$ for all \lr{$n\geq 0$}. 
 
 Since the limit group $\Gamma$ is continuous, there is a sequence $(\gamma_n)_{n=n_0}^{\infty}$ with $\gamma_n\in\Gamma_n\setminus \{\id\}$ such that $\gamma_n\to\id$ as $n\to\infty$. We may assume that $n_0$ is chosen sufficiently large
 to ensure that $\gamma_n(D_R)\subset D_{R+1}$ for $n\geq n_0$.
 Set $\gamma_n^* \defeq \Psi_n(\gamma_n)$. Recall that
$G_n \circ \gamma_n = \gamma_n^*\circ G_n$, and in particular
$\gamma_n^*(0) = G_n(\gamma_n(0))$. Since 
$G_n\to \id$ locally uniformly on $\DD$, for sufficiently large $n$, 
$G_n$ is injective on $D_{R+1}$. 
In particular, $G_n(\gamma_n(0))\neq G_n(0)=0$ for sufficiently large $n$, and
hence $\gamma_n^*\neq \id$. 

Now let $n$ be so large that, for every $z\in D_R$, the arc $\alpha^*(z)$ of the leaf of \lr{$\limitF$} 
 connecting $G_n(z)$ and $G_n(\gamma_n(z)) = \gamma_n^*(G_n(z))$ is 
 contained in $G_n(D_{R+1})$. Its preimage $\alpha(z)$ under $G_n|_{D_{R+1}}$ is an
 analytic arc on a leaf of the foliation $\discF_n$ connecting 
 $z$ and $\gamma_n(z)$. It follows that this leaf does not contain any singularity of the foliation, and that
 the projection
 $\pi_n(\alpha(z))$ is a simple closed curve in $U_n$ (possibly covered several times by $\alpha(z)$), which is 
 a leaf of $\downF_n$. \lr{It is homotopically non-trivial because it does not 
 lift injectively to the universal covering.} By choice of $R$, this proves the first claim. 
 
 Since $K$ is connected, the union $A_n(K)$ of all the leaves of $\downF_n$ intersecting $f_0^n(K)$ is connected. Since $K$ is not contained in a single leaf, \lr{the same is true for 
 $f_0^n(K)$. It follows that $A_n(K)$} is a closed annulus bounded by two closed leaves of $\downF_n$.
 That the modulus of this annulus tends to infinity follows easily from the fact that $\gamma_n\to \id$. Moreover,
 let $\hat{A}_n(K)$ denote the union of all leaves of $\discF_n$ that intersect $D_R\cap \pi_n^{-1}(A_n(K))$. It follows from the above that $\pi_n\colon \hat{A}_n(K)\to A_n(K)$ is a universal covering, and that $g_n$ maps 
 $\hat{A}_n(K)$ injectively onto $\hat{A}_{n+1}(K)$. This implies that $f_n\colon A_n(K)\to A_{n+1}(K)$ is a finite-degree covering map of annuli.
\end{proof}
  
We now prove Theorems \ref{thm:EI} and \ref{thm:semi-contracting}. 
\begin{proof}[Proof of Theorems~\ref{thm:EI} and~\ref{thm:semi-contracting}]
First, to prove Theorem \ref{thm:EI}, assume that $F$ is infinitesimally eventually isometric and essentially thin. Then, by Corollary \ref{cor:doublyconn}, $U_n$ is doubly connected for all large $n$, and by Theorem \ref{thm:annuli}, $\Mod U_n\to\infty$ as $n\to\infty$. The existence and  dynamical properties of the foliations are guaranteed by Theorems \ref{thm:foliations} and \ref{thm:dynfoliations}\lr{. 
(Recall that, if $n_0$ is sufficiently large, then
$\Gamma_{n}\subset \Gamma$ for $n\geq n_0$. No
two points lying on a common leaf of $\discL_n=\limitL$ are identified by an element of $\Gamma$. This, together 
with the characterisation in Theorem~\ref{thm:dynfoliations}, proves that 
$\downL_0$ is indeed an eventually isometric foliation in the sense of Definition~\ref{def:foliations}.) Uniqueness of the foliations also follows from Theorem~\ref{thm:dynfoliations}, and the remark after its proof. Indeed, if $\downF_0'$ is
another contracting foliation on $U_0$, then every leaf of $\downF_0'$ is contained in a leaf of 
$\downF_0$. If $\downF_0'$ is also harmonic, then $\downF_0=\downF_0'$ by
the identity theorem. In the general case, where $\downF_0'$ is merely a topological singular foliation, the same conclusion
follows from topological considerations. 
Indeed, consider the $m$-pronged local model $\mathcal{Q}_m$ defined
in Section~\ref{sec:prelim}, and suppose that $\mathcal{Q}'$ is a singular foliation of its domain $Q_m$ such that every leaf of $\mathcal{Q}'$ is contained in a leaf of $\mathcal{Q}_m$. Then it follows easily that $\mathcal{Q}' = \mathcal{Q}_m$. 
(We omit the details.) For each chart of $\downF_0$, this implies that
local leaves of $\downF_0$ are local leaves of $\downF_0'$. Hence  the two foliations are equal.   
The same argument proves uniqueness of the eventually isometric foliation.}

 Now let us prove Theorem~\ref{thm:semi-contracting}. That $\delta_n(z)\to 0$ follows from Lemma~\ref{lem:homomorphisms}. 
 To construct the annuli $A_n$, let $(K_j)_{j=0}^{\infty}$ be an increasing sequence of compact and connected subsets of
  $U_0$ such that $U_0 = \bigcup K_n$ and such that $K_0$ is not contained in a single leaf of $\downF_0$.
  (For example, let $K_j$ be the \lr{closed} hyperbolic disc in $U_0$ of radius $1+j$ around a base point $z_0\in \lr{U_0}$.)
  
  We may choose a \lr{non-decreasing sequence $(j_n)_{n=N}^{\infty}$ of indices},
  where $N=n_0(K_0)$ and $j_n\to\infty$ as $n\to\infty$, such that $n\geq n_0(K_{j_n})$ for all $n\geq N$.
  (Here $n_0(K)$ is as in Theorem~\ref{thm:annuli}.) We can now define the necessary annuli as $A_n \defeq \lr{\interior(}A_n(K_{j_n})\lr{)}$. The stated
  properties of $A_n$ follow from Theorem~\ref{thm:annuli}, and the existence of the \lr{contracting} foliation 
  and its properties are again given by Theorems \ref{thm:foliations} and \ref{thm:dynfoliations}. 
\end{proof}

\begin{rmk}
 When the system $F$ arises from a wandering domain of a transcendental entire or meromorphic function $f$ (or even a general Ahlfors islands map), the parabolic case cannot arise for eventually isometric domains. Indeed, otherwise the wandering domain $U_n$ would be a punctured disc for sufficiently large $n$, but this is impossible since the Julia set of $f$ has no isolated points.
 \end{rmk}

\newcommand{\GO}{\operatorname{GO}}

We conclude the section by discussing grand orbit classes. Let 
 $z\in U_0$. The \emph{grand orbit $\GO_0(z)=\GO(z,F)$ of $z$ in $U_0$ under the system $F$} consists of all points 
   $w\in U_0$ for which there exists $k\geq 0$ such that \lr{$f_0^k(z)=f_0^k(w)$}. The following
     shows that grand orbits for a non-contracting system are
     either all discrete or all indiscrete, depending on whether the system is
     essentially thick or essentially thin. Moreover, these grand orbits are always nowhere dense. As mentioned
     in the introduction, neither of these facts is true for contracting systems.

\begin{prop}\label{prop:GOrelations}
 Suppose that $F$ is infinitesimally eventually isometric or infinitesimally 
 semi-contracting, and let $z\in U_0$. Then 
    \begin{equation}\label{eqn:grandorbits} \cl_{U_0}(\GO_0(z)) \defeq \overline{\GO_0(z)}\cap U_0 = \pi_0\lr{\bigl(} G_0^{-1}( \Gamma(G_0(\zeta)))\lr{\bigr)},
    \end{equation}
where $\zeta\in \pi_0^{-1}(z)$.

In particular, if $F$ is essentially thick, then the grand orbit
 $\GO_0(z)$ is \lr{discrete.}   
 If $F$ is essentially thin, then for every $z\in U_0$, the grand orbit $\GO_0(z)$ is not discrete but nowhere dense.
\end{prop}
\begin{rmk}
 \lr{In the essentially thin case, the} equality~\eqref{eqn:grandorbits} implies that $\cl_{U_0}(\GO_0(z))$ is locally a leaf of $\downF_0$ in the following sense. 
   If $w\in \cl_{U_0}(\GO_0(z))$
   \lr{, then} 
   there is a neighbourhood $V\subset U_0$ of $w$ 
   such that $V \cap \cl_{U_0}(\GO_0(z)) = V\cap L(w)$, where $L(w)$ is the leaf of \lr{$\downF_0$} containing $w$. \lr{This intersection can be chosen to be a single analytic
   curve (if $w$ is a regular point), or the union of $w$ and  
   an even number of analytic arcs ending at $w$ (if $w$ is singular).}
\end{rmk}
\begin{proof}
 Let $z,w \in U_0$ and let $\zeta,\omega \in \lr{\DD}$ be preimages under the universal covering map $\pi_0$. Then $w\in\GO_0(z)$ if and only if  there is \lr{$k\geq 0$} and \lr{$\gamma\in\Gamma_k$} such that \lr{$g_0^k(\zeta) = \gamma(g_0^k(\omega))$}. It follows that 
  \[ \GO_0(z) = \pi_0\left(G_0^{-1}\left( \bigcup_{k=0}^{\infty} \Psi_k(\Gamma_k)(\lr{G_0(\zeta)}) \right)\right).\]
\lr{Indeed, the} inclusion ``$\subseteq$'' is immediate from the above characterisation of $\GO_0(z)$ 
 and the properties of $\Psi_k$ and $G_0$. To show the inclusion ``$\supseteq$'', 
 \lr{let $k\geq 0$, $\omega\in \DD$ and $\gamma_{k}\in \Gamma_{k}$ such that 
   $\Psi_k(\gamma_k)(G_0(\omega)) = G_0(\zeta)$. For $j\geq k$, define
   $\gamma_j$ by $\gamma_{j+1}\defeq \psi_j(\gamma_j)$; then
   $\Psi_j(\gamma_j)=\Psi_k(\gamma_k)$ for all $j$. It follows that 
   $G_j\lr{\bigl(}\gamma_j(g_0^j(\omega))\lr{\bigr)} = G_j(g_0^j(\zeta)\lr{)}$. Recall} 
that $G_k\to\id$ \lr{uniformly on compact sets} as $k\to\infty$. \lr{The 
  points $g_0^j(\omega)$ remain within a bounded hyperbolic distance of $0$ 
  as $j\to\infty$, and $\gamma_j\to \Psi_k(\gamma_k)$ by construction of
  $\Psi_k$. Hence 
  $\gamma_j(g_0^j(\omega)) = g_0^j(\zeta)$ for sufficiently large $j$. The 
  claim follows.}

 We have 
   \[ \cl_{\DD}\left(\bigcup_{k=0}^{\infty} \Psi_k(\Gamma_k)(G_0(\zeta))\right) = \Gamma(G_0(\zeta)). \]
   Since $G_0^{-1}(\Gamma(G_0(\zeta)))$ is invariant under $\Gamma_0$, its image under $\pi_0$ is closed, and 
   hence~\eqref{eqn:grandorbits} follows. 

   If $F$ is essentially thick, then $\Gamma$ is discrete, and $G_0^{-1}(\Gamma(G_0(\zeta)))$ is a discrete subset of $\DD$. Again, by invariance under
   $\Gamma_0$, its image under $\pi_0$ is also discrete. 
    If $F$ is essentially thin, then 
    $\Gamma(G_0(\zeta))$ is a leaf of the foliation \lr{$\limitF$}, and the same reasoning shows that $\cl_{U_0}(\GO_0(z))$ is non-discrete but nowhere dense.
\end{proof}

\section{Comparison with multiply connected domains of entire functions}\label{sec:entire}

In this section, we discuss our results in the context of the best understood multiply connected wandering domains, those of entire functions, while also generalising to other wandering domains with similar internal dynamics. Multiply connected wandering domains of entire functions were first constructed and studied by Baker. Already T\"opfer~\cite{toepfer1939}*{p.~67} observed that any multiply connected
 Fatou component is necessarily escaping, while Baker~\cites{Bak63,Bak76} constructed the
 first examples of transcendental entire functions with
 multiply connected wandering domains. He also showed~\cite{Bak84}*{Theorem 3.1} that, for any
 such wandering domain $U$, the Fatou component $U_n$ containing $f^n(\lr{U_0})$ must
 separate $U_{n-1}$ from $\infty$ for large $n$. In particular, 
 all $U_n$ are bounded domains of the plane. Later, Kisaka and Shishikura \cite{KS08} constructed the
 first example of a doubly connected wandering domain and 
 showed that the eventual connectivity of a multiply connected wandering domain of a transcendental entire
 function is always well-defined, and is either two or infinity. 
 
 In terms of the internal dynamics of these wandering domains, however, the greatest advances were made by Bergweiler, Rippon, and Stallard \cite{BRS13}. They  showed that the dynamical behaviour of a multiply connected wandering domain $U$ of an entire function $f$ can be described using a dynamically defined sequence of harmonic functions $u_n\colon U_n\to(0, +\infty)$. 
 In~\cite{Fer22}*{Theorem~1.2}, it was shown that the 
 \lr{path-connected components of the} 
 level sets of \lr{these harmonic functions} form
 the leaves of a contracting foliation. We begin by noting that the contracting foliations constructed
 in the previous section can also be described \textit{a posteriori} as level sets of
 harmonic functions. 

Throughout this section, $F = \lr{(U_n,f_n)}_{n=0}^\infty$ is an essentially thin and non-contracting system (of which multiply connected wandering domains of entire functions are a special case). The notation and terminology are still the same as in Section \ref{sec:global}.

\begin{thm}[Harmonic functions for contracting foliations]\label{thm:harmonic}
 Let $F = \lr{(U_n,f_n)}_{n=0}^\infty$ be an essentially thin and non-contracting system of
 holomorphic maps between hyperbolic Riemann surfaces.
 
 Then there exists a sequence $u_n\colon U_n\to (0,\infty)$ of positive harmonic functions such that $u_n = u_{n+1}\circ f_n$ for all $n$, and such that the foliation of $U_n$ by the 
 \lr{path-connected components of the} level sets of $u_n$ is the contracting foliation of $U_n$.

  The \lr{sequence $(u_n)_{n=0}^{\infty}$} is unique up to post-composition by a
   \lr{non-constant} real-affine map $\RR\to\RR$, and can be chosen such that its range 
  \lr{$\bigcup_{n=0}^{\infty} u_n(U_n)$} is either $(0,1)$ (in the hyperbolic case) or $(0,\infty)$ (in the parabolic case).
\end{thm}
\begin{proof}
 We use the results and notation from Section~\ref{sec:global}. In particular, let
 $\limitF$ be the foliation of $\DD$ by the limit group $\Gamma$. Recall that, for $n\geq 0$,\ 
 the contracting foliation $\downF_n$ of $U_n$ is the push-forward under the projection $\pi_n$ of 
 $\discF_n = G_n^{-1}(\limitF)$, where $G_n\colon \DD\to \DD$ are the functions
 given by Lemma~\ref{lem:limitfunctions}. 
  
  Let $\phi$ be a conformal isomorphism from $\DD$ onto the upper half-plane
 (in the parabolic case) or the strip $\{a+ib\colon b\in (0,1)\}$ (in the hyperbolic case)
 that takes the leaves of the foliation $\limitF$ by orbits of the limit group
 $\Gamma$ to horizontal lines. Define
 $\tilde{u}_n\colon \DD\to (0,\infty); z\mapsto \im(\phi(G_n(z)))$. If
 $\gamma\in\Gamma_n$, then 
  \[ \tilde{u}_n(\gamma(z)) = \im\lr{\bigl(}\phi (G_n(\gamma(z)))\lr{\bigr)} = \im\lr{\bigl(}\phi (\Psi_n(\gamma)(G_n(z)))\lr{\bigr)}. \]
  Since $\Psi_n(\gamma)\in \Gamma$, the points $G_n(z)$ and $\Psi_n(\gamma)(G_n(z))$ belong to the
  same leaf of $\limitF$, and hence
  \[ \tilde{u}_n(\gamma(z)) =
     \im(\phi(G_n(z))) =
          \tilde{u}_n(z). \]
  We conclude that
    \[ u_n\colon U_n\to (0,\infty); \quad u_n(\pi_n(z)) \defeq \tilde{u}_n(z) \]
    is a well-defined harmonic function on $U_n$ with the desired properties. 

 Two harmonic functions $u$ and $v$ on a plane domain have the same family of level sets if and only if they differ by post-composition with a \lr{non-constant} real-affine map. Indeed, near a regular point, we may write $v(z)=\alpha(u(z))$ for a strictly monotone function
 $\alpha$. The chain rule, together with the fact that $u$ and $v$ are harmonic, implies that
 $\alpha'' = 0$. So $\alpha$ is affine, and by the identity theorem $v(z)=\alpha(u(z))$ on $U_n$. This proves
 \lr{uniqueness up to post-composition}.

 \lr{The range of $u_n$ is the range of $\tilde{u}_n$, which tends to
    $\im \phi(\DD)$ as $n\to\infty$. The latter is either 
    $(0,1)$ or $(0,\infty)$, establishing the final claim.}
\end{proof}

In the case of a multiply connected wandering domain $U_0$
 of an entire function $f$, in~\cite{BRS13} the harmonic function $u_0\colon U_0\to (0,\infty)$ 
  is obtained as a limit 
   \begin{equation}\label{eqn:BRSu} u_0(z) = \lim_{n\to\infty} \frac{\log \lvert f^n(z)\rvert}{\log \lvert f^n(z_0)\rvert},\end{equation}
   where $z_0\in U_0$ is a base point. The function takes values in an interval
   $(a,b)$ with $0\leq a < 1 < b \leq \infty$. The exact interval depends on the choice of
   the base point $z_0$; a change of base point changes the function $u_0$ by a positive
   multiplicative constant, and in particular the ratio $b/a\in (1,\infty]$ is 
   independent of $z_0$. 

 Now let $F=\lr{(U_n,f_n)}_{n=0}^{\infty}$
  be a thin and non-contracting system as in the theorem; we continue to use the
  notation from Section~\ref{sec:global}
  \lr{and write $\alpha_n(z)$ for the leaf of $\downF_n$ containing
  $f_0^n(z)$}. Let $z_0\in U_0$. 
  In order to compare the construction from~\cite{BRS13} 
  to the functions $u_n$ in Theorem~\ref{thm:harmonic}, let us suppose that
  each $U_n$ is a subset of the Riemann sphere. For simplicity, we may additionally assume that 
   \begin{enumerate}[(a)]
   \item $z_n=1$ for all $n\geq 0$;
   \item $U_n\subset \Chat\setminus\{0,\infty\}$;
   \item the leaf $\alpha_n(z_0)$ 
   \lr{separates}
    $0$ from $\infty$ for sufficiently large  $n$; 
    \item if $u_0(z)>u_0(w)$, where $z,w\in U_0$, then for sufficiently large $n$, $\alpha_n(z)$ separates
    $\alpha_n(w)$ from $\infty$.\label{item:orientation}
   \end{enumerate}
 (These can be achieved by applying
  M\"obius transformations to the domains $U_n$, and post- and pre-composing $f_{n-1}$ and $f_n$ accordingly.)

 Let $r_n \in [0,1]$ be minimal such that 
    \[ \{z\in\Cx\colon \lr{r_n}<\lvert z\rvert < 1 \} \subset U_n.\] 
   Our normalisation \lr{ensures} that, for sufficiently large $n$, we have $0<r_n<1$.
   Moreover, for large $n$, $U_n$ contains
   an annulus of large modulus containing $z_n=1$ and separating $0$ from $\infty$. Hence $r_n\to 0$ as $n\to\infty$. 
   \lr{For sufficiently large $n$,} define 
    \begin{equation}\label{eqn:hn} h_n\colon U_0\to \RR; \quad  
        z\mapsto \frac{\log \lvert f_0^n(z)\rvert}{\lvert \log r_n\rvert}.\end{equation}

 \begin{prop}[Limit formula for $u_0$]\label{prop:harmonic}
    \lr{Let $(u_n)_{n=0}^{\infty}$ be as in Theorem~\ref{thm:harmonic}, normalised so that 
    $\bigcup_{n=0}^{\infty} u_n(U_n)$ is either $(0,1)$ or $(0,\infty)$. Let
    $h_n$ be defined by~\eqref{eqn:hn}. 
    
    Then} $h_n(z)\to u_0(z)/u_0(z_0) - 1$ locally uniformly \lr{on $U_0$} as $n\to\infty$.
 \end{prop}
 \begin{proof}[Sketch of proof]
  \lr{For sufficiently large $n$, $U_n$ is not simply connected; let} $\gamma_n\in\Gamma_n\lr{\setminus\{\id\}}$ be an element minimising $\lvert \gamma_n(0)\rvert$.
    Since  $\Gamma_n\to \Gamma$, \lr{if $n$ is large enough,}
    there are exactly two such choices, which are inverse to each other. 
    Which choice we take will be specified below. 

   If $\gamma_n$ is hyperbolic, let $\phi_n$ be a conformal isomorphism from $\DD$ to 
    a strip $V_n=\{z\in\Cx\colon -1 < \im z < b_n\}$, with $b_n>0$, such that $\phi_n(0)=0$ 
    \lr{and}
    $\gamma_n$ is conjugate, via \lr{$\phi_n$}, to a translation $\tau_n\colon z\mapsto z + t_n$ by a positive real number $t_n$. 
    (To achieve this, map the two fixed points of $\gamma_n$ to the two ends of a horizontal strip,
    and post-compose by a real-affine map to achieve the desired normalisation.) 
    Replacing $\gamma_n$ by $\gamma_n^{-1}$, if necessary, we can ensure additionally that the
    image of the real axis under $\pi_n\circ \phi_n^{-1}$, which is a simple closed curve in 
    $U_n$ separating $0$ from $\infty$, is \lr{negatively} oriented.
    \lr{(So the image of a point with small negative imaginary part is separated from 
    $\infty$ by this curve.)}

 If $\gamma_n$ is parabolic\footnote{As remarked by Beardon and Maskit \cite{BM74}*{Lemma 1} (see also \cite{Oht54}*{Theorem 2}), parabolic deck transformations correspond to isolated punctures of the Riemann surface. Hence
 the group $\Gamma_n$ never contains parabolic elements when $U_n$ is a  wandering domain of a meromorphic function,
 or even an Ahlfors islands map.}, we similarly let $\phi_n$ be a conformal isomorphism from
  $\DD$ to the half-plane  
       $\{z\in\Cx\colon \im z > -1 \}$, mapping the unique fixed point of $\gamma_n$
       to $\infty$ and fixing $0$. Then $\gamma_n$ corresponds to a real translation
       $\tau_n$;
       replacing $\gamma_n$ by $\gamma_n^{-1}$, we may assume that the  translation constant
       $t_n$ is positive. 
       Condition~\ref{item:orientation} ensures that (for sufficiently large $n$),
       again the image of the real axis under $\pi_n\circ \phi_n^{-1}$ is \lr{negatively} oriented.

  It is not difficult to show that $t_n\to 0$, and that the 
   foliations $\phi_n(\discF_n)$ converge to the horizontal foliation on a limiting
   strip or half-plane. Now, the domains $U_n$ contain a topological annulus of large modulus, which means (by a theorem of Teichm\"uller; see also \cite{AW2009}*{Theorem 3.17}) that $U_n$ contains a \textit{round} annulus of comparable modulus for large $n$. 
   \lr{Let $A_n$ be the largest such annulus centred at $0$, and let $\tilde{A}_n$ be its
   preimage component containing $0$ under
   the universal covering $\pi_n\circ \phi_n^{-1}$. Then
   $\tilde{A}_n$ is invariant under translation by $t_n$, and contains a horizontal
   strip around the real axis, again for large $n$. Let $B_n$ be the 
   preimage of $A_n$ under the universal covering map 
   $E(\zeta)\defeq \exp(-i\zeta\cdot 2\pi/t_n)$. Let
   $\psi_n\colon B_n\to \tilde{A}_n$ be the conformal isomorphism 
   that fixes $0$ and commutes with addition by $t_n$. By Carath\'eodory kernel
   convergence \cite{Hej74}, we have $\psi_n\to\id$ locally uniformly. Write 
   $f_0^n(z) = E(\zeta_n)=\pi_n(\phi_n^{-1}(\omega_n))$ for $z\in U_0$ and
   sufficiently large $n$, with $\zeta_n\in B_n$ and $\omega_n=\psi_n(\zeta_n)$. 
   We have $\im \omega_n \to u_0(z)/u_0(z_0) - 1$ as $n\to\infty$, locally uniformly
    in $z$. So 
    \[\ \frac{t_n \log \lvert f_0^n(z)\rvert}{2\pi} - \left(\frac{u_0(z)}{u_0(z_0)}-1\right)= 
         (\im \zeta_n - \im \omega_n) + \left(\im \omega_n - \frac{u_0(z)}{u_0(z_0)}+1\right)  \to 0\]
    locally uniformly in $z$. The Teichm\"uller modulus theorem
    shows that $2\pi/t_n = \lvert \log r_n\rvert + O(1)$; since 
    $\lvert \log r_n\rvert \to \infty$, $t_n \lvert \log r_n\rvert / 2\pi \to 1$.}
 \end{proof}

 Our normalisation above has assumed prior knowledge of the foliations $\downF_n$ and the functions $u_n$. It is not difficult to restate Proposition~\ref{prop:harmonic} in such a way  as to avoid this; however, its
 proof would still rely on \lr{the} existence of the foliation. 
  In contrast, in the entire case the \lr{foliation} is obtained as a consequence of the convergence of the functions.  

 As mentioned in the introduction, the proof of convergence in the entire case relies on properties of the maximum modulus, particularly
  Wiman--Valiron theory (which is available only for entire functions and certain meromorphic functions; see
  \cite{BRS08}). The ratio $b/a$ mentioned above is \emph{external}: it depends on the
  way in which the domains $U_n$ tend to infinity in the plane, rather than on their intrinsic geometry. Therefore no 
  analogue of this ratio exists in the general meromorphic setting.
  However, whether this ratio is finite or infinite is an \emph{internal} property: it distinguishes 
  whether the limit group $\Gamma$ is hyperbolic or parabolic. 

 Another difference in the entire case is that each $U_n$ has a distinguished \textit{outer boundary component} $\partial_{\out} U_n$: the component of $\partial U_n$ that separates $U_n$ from $\infty$ in $\Cx$. The entire function must map $\partial_{\out} U_n$ to $\partial_{\out} U_{n+1}$. Moreover, the outer boundary component is also distinguished by the internal dynamics: It is the only place near which the function $u_n$ may take values close to $b$; indeed, $u_n$ extends continuously to every $\zeta\in\partial U_n\setminus \partial_{\out} U_n$ with $u_n(\zeta)=a$~\cite{BRS13}*{Theorem~1.6~(b)}. Moreover, the role of the outer boundary
 differs, depending on whether $b<\infty$ or $b=\infty$: In the former case, $u_n(z)$ is (up to rescaling by a real-affine map) exactly given by the harmonic measure of $\partial_{\out} U_n$ from $z$, while
 otherwise $\partial_{\out} U_n$ has zero harmonic measure~\cite{BRS13}*{Theorem~1.6~(d)}.

There is no distinguished outer boundary component in our more general setting. Nonetheless, some \lr{versions} of the 
 above results still hold for essentially thin non-con\-tracting systems, if we impose the additional assumption
 that the maps $f_n$ are \emph{proper maps}; that is,
 the preimage of a compact subset of $U_{n+1}$ 
 under $f_n$ is compact. This condition is always satisfied for orbits of bounded wandering domains of transcendental meromorphic functions. In particular, it holds for multiply connected wandering domains of
 entire functions. 
 
We will use the notion of the set $E(X)$ of \emph{ends} of a surface $X$. In the case where $X$ is a subset of the sphere, the ends of $X$ are exactly the connected components of $\partial X$; for this reason, the ends of a surface are also called its \emph{boundary components}. (See~\cite{richardsclassification}*{Definition~1}
or~\cite{Oht54}*{Section~I.1}.) An \emph{end} $e$ of $X$ is determined by a choice 
  $e(K)$ \lr{of} a connected component of $X\setminus K$ for every compact subset
  $K\subset X$, with the property that $e(K_1)\subset e(K_2)$ when $K_2\subset K_1$. 
  
  The set of ends is non-empty if and only if $X$ is non-compact. 
  $X\cup E(X)$ is a compact connected topological (even metrisable)
  space; a neighbourhood base of $e\in E(X)$ is given by the domains
  $e(K)\lr{\cup \{\tilde{e}\in E(X)\colon \tilde{e}(K)=e(K)\}}$. 
  A proper map $f\colon X\to Y$ between surfaces extends continuously
  to a map $f\colon E(X)\to E(Y)$: the image $\tilde{e}=f(e)\in E(Y)$ of
  $e\in E(X)$ is defined by \lr{letting $\tilde{e}(K)$ be} the connected
  component of $Y\setminus K$ containing $f(e(f^{-1}(K)))$.
 
 Now let $F=\lr{(U_n,f_n)}_{n=0}^{\infty}$ be as \lr{above}, and 
  define $\partial_{-} U_n$ to consist of all ends $e$ of $U_n$ for which there exists a curve 
  $\gamma$ tending to $e$ along which $u_n$ tends to $0$. Similarly, define $\partial_{+} U_n$ to
  consist of all ends $e$ of $U_n$ for which there is a curve $\gamma$ along which $u_n$ tends to 
  $1$ (in the hyperbolic case) or $\infty$ (in the parabolic case). Here $u_n$ is the function from Theorem~\ref{thm:harmonic}. Observe that, in the hyperbolic case, replacing $u_n$ by $1-u_n$ will exchange the roles of 
  $\partial_+ U_n$ and $\partial_- U_n$.
 
 \begin{prop}[Boundary components]\label{prop:boundarycomponents} Let $F = \lr{(U_n,f_n)}_{n=0}^\infty$ be an essentially thin and non-contracting system of
 holomorphic maps between hyperbolic Riemann surfaces, and assume that all maps $f_n$ are proper. \lr{Then} $f_n(\partial_- U_n) = \partial_-(U_{n+1})$ and $f_n(\partial_+ U_n) = \partial_+ U_{n+1}$ \lr{for every $n\geq 0$.}

  \lr{Moreover, suppose that there exists $n_0\geq 0$ such that $U_{n_0}$ has genus $0$. Then $U_n$ has genus $0$ for $n\geq n_0$, and }
   the sets $\partial_-U_n$ and $\partial_+ U_n$ are disjoint \lr{for all $n\geq 0$.} 
 \end{prop}
 \begin{proof}
     The \lr{first} claim follows  from the fact that \lr{$f_n$} is proper, together
     with the functional
     relation $u_n=u_{n+1}\circ f_n$. 
     
     To prove the \lr{second}  statement, we observe that all regular leaves of the contracting
     foliation are closed analytic curves. Indeed, every such leaf \lr{$L$} maps to a closed 
     \lr{regular} leaf \lr{$\tilde{L}$} by $f_n^k$ for sufficiently large $k$ by Theorem~\ref{thm:annuli}. 
     \lr{Since $L$ is regular, it contains no critical point of $f_n^k$. 
     Since $f_n^k$ is proper, it maps $L$ to $\tilde{L}$ as a covering map, and hence
     $L$ is a simple closed curve.} 
     
     Next, \lr{suppose that $U_n$ has genus zero; i.e., it is conformally
     equivalent to a subdomain of $\Chat$. We} 
    claim that, for every end $e$ of $U_n$ and every compact connected
     subset $K\subset U_n$, 
     there is a simple closed
      leaf $\lambda_n(K)$ of the foliation that separates $e$ from $K$.
     This \lr{follows} from Theorem~\ref{thm:annuli}. 
     Indeed, $f_n^k(K)$ is eventually contained in an annulus in $U_{n+k}$ bounded by two simple closed leaves
     of the foliation $\downF_{n+k}$. We may choose these boundary leaves so that they do not contain critical values of $f_n^k$. The preimage of the annulus is compactly contained in $U_n$, and one of its finitely many
     boundary curves
     has the desired property. 

    Let $e\in \partial_- U_n$ and let $\gamma$ be a curve tending to $e$, starting at $z_n = f_0^n(z_0)$,
    such that $u_n\to 0$ along $\gamma$. For $\rho>0$, consider the \lr{closed}
    hyperbolic disc
    $K_n(\rho)\defeq \lr{\overline{D_{U_n}(z_n,\rho)}}$. The curve $\gamma$ must intersect the leaf $\lambda_n(K_n(\rho))$ for all 
    $\rho>0$. Recall that $u_n$ takes a  constant value on $\lambda_n(K_n(\rho))$, which must tend to zero
    as $\rho\to \infty$. In particular, there can be no curve tending to $e$ along which
    $u_n\to b\in (0,\infty]$, and $e\notin \partial_+ U_n$.

   \lr{This proves disjointness of $\partial_- U_n$ and $\partial_+ U_n$ whenever
    $U_n$ has genus $0$. It is a classical fact that a non-constant
    proper holomorphic map
    $f\colon X\to Y$
    between Riemann surfaces does not increase genus. (For compact
    surfaces, this is a consequence of the Riemann--Hurwitz formula. For 
    non-compact surfaces, first restrict $f$ to
    the preimage of a finitely bounded sub-surface of $Y$. Then reduce to the compact case by
    filling in punctures and holes. We omit the details;
    for the finitely connected case, see also~\cite{LyzzaikStephenson1991}*{Lemma~1}.) So, if $U_{n_0}$ has
    genus $0$, then the same holds for $U_n$ with $n\geq n_0$. If $n<n_0$, then
    disjointness of $\partial_- U_n$ and $\partial_+ U_n$ follows from disjointness
    for $n_0$ and the first
    part of the proposition.}
 \end{proof}

Following~\cite{BRS13}, we now wish to discuss
 harmonic measure on the ends of $U_n$. Here the 
 \emph{harmonic measure} $\omega_X(z_0,E)$ 
  of a set $E\subset E(X)$ with respect to a point $z_0$ in a hyperbolic
  Riemann surface $X$ can be described as the
  probability that
  Brownian motion starting at $z_0$ converges to one
  of the ends $e\in E$. Equivalently, let $\pi\colon \DD\to X$ be a 
  universal covering map with $\pi(0)=z_0$. Then $\omega_X(z_0,E)$ is the
  normalised Lebesgue measure of the \lr{set} of $\alpha\in\partial\DD$
  for which $\pi(z)\to E$ as $z\to \alpha$ non-tangentially. 
  (See~\cite{Tsu75}*{Section X.2} for an introduction to
  harmonic measure on Riemann surfaces, and an alternative definition.)
  
  The harmonic measure of the boundary is 
  non-zero if and only if there exists 
  a Green's function on $X$~\cite{Tsu75}*{Theorem~X.1}. 
  In the latter case
  we say, following Tsuji~\cite{Tsu75}*{Section~X.1}, that
  $X$ has \emph{positive boundary}. (Often, the term
  \emph{non-polar boundary} is also used, particularly when
  $X\subset\Chat$.) This condition always holds for
  our surfaces $U_n$.

\begin{prop}[Positive boundaries]\label{prop:nonpolar}
    Let $F = \lr{(U_n,f_n)}_{n=0}^\infty$ be an essentially thin and non-contracting system of
 holomorphic maps between hyperbolic Riemann surfaces. 
 Then every $U_n$ has positive boundary. 
 
 In particular, if $n\geq 0$ and 
 $\pi_n\colon \DD\to U_n$ is a universal covering
 map, then for almost every $\alpha\in\lr{\partial \DD}$, $\pi_n(z)$ tends to an end of
 $U_n$ as $z\to \alpha$ non-tangentially. 
\end{prop}
\begin{proof}
  \lr{For a Riemann surface $X$, positive} boundary, i.e., the existence of a Green's 
  function, \lr{follows from} 
   the existence of a non-constant positive harmonic function
   on $X$;
   see~\cite{Tsu75}*{Theorem~X.7}. \lr{The latter} holds for
   \lr{$X=U_n$} by Theorem~\ref{thm:harmonic}.
   
 The second statement follows from the first in two \lr{additional} steps. First, by \cite{Tsu75}*{XI.22}, there exist full-measure subsets $E_n\subset\partial\DD$ such that, for $\zeta\in E_n$, $\pi_n(r\zeta)$ converges to some end of $U_n$ as $r\nearrow 1$. For universal coverings, having a radial limit at \lr{$\zeta$} is well-known to be equivalent to having a non-tangential limit. Indeed, points in a Stolz angle at $\zeta$
 are a bounded hyperbolic distance away from the
 radial ray $\{r\zeta\colon r\in [0,1)\}$.
 By the Schwarz--Pick lemma, the same is true for
 their images under $\pi_n$.  We refer the reader to \cite{FJ25a}*{Proposition 4.2} for details.
\end{proof}

\begin{thm}[Harmonic measure of boundary components]
  Let $F = \lr{(U_n,f_n)}_{n=0}^\infty$ be an essentially thin and non-contracting system of holomorphic maps between hyperbolic Riemann surfaces. Additionally assume that all maps $f_n$ are proper.  

 Then, for all $n\geq 0$  and all $z\in U_n$,
   \begin{equation}\label{eqn:harmonicmeasure} \omega_{U_n}\left(z, \partial_+ U_n \cup \partial_- U_n\right)=1\lr{;} \end{equation}
  \lr{if the limit group $\Gamma$ is parabolic, even
   \begin{equation}\label{eqn:parabolicfullmeasure} \omega_{U_n}\left( z, \partial_- U_n\right) = 1. \end{equation}}

  \lr{If $\Gamma$ is hyperbolic, assume additionally that 
  \begin{equation}\label{eqn:intersectionmeasure}
  \omega_{U_n}\left(z, \partial_+ U_n \cap \partial_- U_n\right)=0\end{equation}
  for some, and hence all, $z\in U_n$. Then} 
   \begin{equation}\label{eqn:hyperbolicmeasure} u_n(z) = \omega_{U_n}\left( z, \partial_+ U_n\right) \end{equation}
   for all $z\in U_n$.
\end{thm}
\begin{proof}
Set $b \defeq 1$ in the hyperbolic case and $b\defeq \infty$ in the parabolic case. Since the functions $f_n$ are proper and each $U_n$ has 
positive boundary, the lifts $g_n\colon \DD\to\DD$ are inner functions; 
that is, for almost every $\theta\in\partial \DD$, the radial limit of $g_n$
(which exists by Fatou's theorem) is again on $\partial\DD$.

Indeed, let $\theta\in\lr{\partial \DD}$ be such that the radial limit of $g_n$ at $\theta$
exists. By Proposition~\ref{prop:nonpolar}, for almost every such $\theta$,
 $\pi_n(z)$ tends to an end \lr{$\pi_n^*(\theta)$} of $U_n$ as $z\to\theta$ non-tangentially
in $\DD$. Since $f_n$
is proper, $f_n(\pi_n(z)) = \pi_{n+1}(g_n(z))$ tends to an end of $U_{n+1}$. This means
that the radial limit of $g_n$ at $\theta$ must indeed lie on the unit circle. (Compare~\cite{FJ25}*{Proposition 2.10}\lr{.}) 

Consequently, by \cite{Fer23}*{Theorem 1.1}, the functions $G_n$ given by Lemma \ref{lem:limitfunctions} are also inner. It follows that the radial limit of $G_n$ at almost every point $\theta\in\partial\DD$ exists,  belongs to $\partial\DD$, and is not a fixed point of the limit group 
$\Gamma$. In particular, by definition of the harmonic function $\tilde{u}_n$ in the proof of Theorem~\ref{thm:harmonic}, the radial limit of $\tilde{u}_n$ exists at almost every $\theta$ and is equal to $0$ or $b$. Moreover, in the parabolic case, this limit is equal to $0$ almost everywhere.

By Proposition~\ref{prop:nonpolar}, for almost every $\theta\in \partial\DD$, the image under $\pi_n$ of the radial \lr{segment}  connecting $0$ and $\theta$ is a curve tending to an end of $U_n$. 
This proves~\eqref{eqn:harmonicmeasure}. In the parabolic case,\lr{~\eqref{eqn:parabolicfullmeasure}} also follows.

\lr{Now suppose that $\Gamma$ is hyperbolic and~\eqref{eqn:intersectionmeasure} holds. Let} $C$ be the subset of $\partial \DD$ bounded by the two fixed points of $\Gamma$ such that $\tilde{u}_n(\zeta)\to 1$ if and only if $G_n(\zeta)\to C$. Then $\tilde{u}_n(\zeta) = \omega_{\DD}(G_n(\zeta),C)$ for all $\zeta\in\DD$. Loewner's lemma (see e.g. \cite{DM91}*{Corollary 1.5}) implies that  $G_n\colon\DD\to\DD$ and its boundary extension $G_n^*$ (defined almost everywhere via radial limits) \lr{preserve} harmonic measure. That is,
    \[ \tilde{u}_n(\zeta) = \omega_{\DD}(G_n(\zeta),C) = \omega_{\DD}(\zeta,(G_n^*)^{-1}(C)). \] 
 
 \lr{By~\eqref{eqn:intersectionmeasure} and the above discussion,}  $(G_n^*)^{-1}(C)$ and $(\pi_n^*)^{-1}\lr{(\partial_+ U_n)}$ agree up to a set
 of harmonic measure zero. Thus~\eqref{eqn:hyperbolicmeasure} follows.
\end{proof}
\begin{rmk}
    If the $U_n$ are subdomains of the sphere, the notion of
    harmonic measure of boundary components, described above,
    coincides with the usual notion of harmonic measure on the 
    boundary of $U_n$.
\end{rmk}

In the entire case, $\partial_+ U_n$ contains only the outer boundary component $\partial_{\out} U_n$, while all other boundary components are in $\partial_- U_n$. In the meromorphic case, both $\partial_+ U_n$ and $\partial_- U_n$ may consist of several boundary components. Moreover, it is plausible that there are examples
in which some ends of $U_n$ do not belong to either
$\partial_+ U_n$ or $\partial_- U_n$. 

\section{Eventual connectivity: The proof of Theorem~\ref{thm:main}}\label{sec:mainproof}

First, we restate Theorem \ref{thm:main} for non-autonomous systems.
\begin{thm}[Internal dynamics of non-autonomous systems on hyperbolic surfaces]\label{thm:main-nonautonomous} Let $F=\lr{(U_n,f_n)}_{n=0}^{\infty}$ be a non-autonomous
  system of holomorphic maps between hyperbolic Riemann surfaces. Then
   one of the following holds.
  \begin{enumerate}[(1)]
    \item \label{item:contractingNA}
     $F$ is globally contracting.  
   \item $F$ is essentially thick and globally semi-contracting. Either
       $\delta_n(z)\to\infty$ for all $z\in U_0$, or for all $z\in U_0$, $\delta_n(z)$ is eventually \lr{non-increasing} and tends to
       a positive finite value as $n\to\infty$.\label{item:semicontractingNA}
    \item  $F$ is essentially thin, infinitesimally semi-contracting and bimodal. Furthermore,
     \lr{$\delta_n(z)$ is eventually non-increasing and}
          $\delta_n(z)\to 0$ for all $z\in U_0$.\label{item:bimodalNA}
    \item $F$ is essentially thick and locally eventually isometric. 
     \lr{For every $z\in U_0$, the sequence $(\delta_n(z))_{n=0}^{\infty}$ is eventually
      constant.} \label{item:locallyisometricNA}
\label{item:nongloballyisometricNA}
    \item $F$ is essentially thin, infinitesimally eventually isometric and trimodal. 
    \lr{The domain $U_n$ is doubly connected 
       for all sufficiently large $n$, and} $\delta_n(z)\to 0$ for all \lr{$z\in U_0$}.\label{item:trimodalNA}
  \end{enumerate}
\end{thm}
\begin{proof}
\lr{If $F$ is infinitesimally contracting, then it is globally
contracting by Corollary~\ref{cor:contracting}. If $F$ is infinitesimally 
semi-contracting and essentially thin, then~\ref{item:bimodalNA} holds 
by~Lemma~\ref{lem:homomorphisms} and Theorem~\ref{thm:dynfoliations}. 
Similarly, if $F$ is infinitesimally eventually isometric and 
essentially thin, then~\ref{item:trimodalNA} holds by~Lemma~\ref{lem:homomorphisms}, Corollary~\ref{cor:doublyconn} and
 Theorem~\ref{thm:dynfoliations}. So
 it only remains to consider the cases where the system is both essentially thick and infinitesimally 
 non-contracting; we assume from now on that $F$ satisfies both of these conditions.}

\lr{Let $(z_n)_{n=0}^{\infty}$ and $(w_n)_{n=0}^{\infty}$ be two orbits of the system $F$, and take 
a preferred lift 
$G=(\DD, g_n)_{n=0}^\infty$ of $F$ via
universal covering maps $(\pi_n)_{n=0}^{\infty}$, where $\pi_n(0)=z_n$ for all $n$. Choose 
$\widetilde{w}_n\in\pi_n^{-1}(w_n)$ with $\dist_{\DD}(0,\widetilde{w}_n) = \dist_{U_n}(z_n,w_n)$. Also define
$\omega_{n+1} \defeq g_n(\widetilde{w}_n)$ for all $n$. Note that 
 \begin{equation}\label{eqn:dist-inequalities} \dist_{U_n}(z_n,w_n) = \dist_{\DD}(0,\widetilde{w}_n) \leq \dist_{\DD}(0,\omega_n) \leq \dist_{\DD}(0,\widetilde{w}_{n-1}) \end{equation}
 for $n\geq 1$. We claim that $\dist_{U_n}(z_n,w_n)\to 0$ as $n\to\infty$ if and only if $z_n=w_n$ for sufficiently large $n$. Indeed, 
  suppose that the former holds. Then $\widetilde{w}_n\to 0$ and $\omega_n\to 0$. For $n\geq 1$, there is $\gamma_n\in\Gamma_n$ with $\gamma_n(\omega_n)=\widetilde{w}_n$. 
  Since $F$ is essentially thick, 
  $\liminf_{n\to\infty} \delta_{n}(w_0) > 0$, and thus 
  $\gamma_n=\id$ for sufficiently large $n$. We conclude that
   $g_n(\widetilde{w}_n)=\widetilde{w}_{n+1}$ for sufficiently large $n$, and hence 
   $G_n(\widetilde{w}_n)=\lim_{k\to\infty} \widetilde{w}_{k}=0=G_n(0)$. 
  But $G_n\to \id$ locally uniformly; in particular,  for any disc $D\Subset\DD$ 
centred at $0$, $G_n|_{D}$ is 
  injective for sufficiently large $n$. 
  So $\widetilde{w}_n=0$ for sufficiently large $n$, which means that $w_n=z_n$, as claimed.}

 \lr{If $F$ is infinitesimally semi-contracting, then for infinitely many $n$,
  $g_n$ is not a conformal isomorphism, and hence 
  \[ \dist_{U_{n+1}}(z_{n+1},w_{n+1})
     = \dist_{\DD}(0,\widetilde{w}_{n+1}) \leq \dist_{\DD}(0,\omega_{n+1}) <
    \dist_{\DD}(0,\widetilde{w}_n) = \dist_{U_n}(z_n,w_n) \]
 if $z_n\neq w_n$. From this and the above, we conclude that either
 $z_n=w_n$ for some $n\geq 0$, or the sequence $\dist_{U_n}(z_n,w_n)$ is not 
 eventually constant and tends to some non-zero limit. So $F$ is
  globally semi-contracting. The claim about injectivity radii holds by Lemma~\ref{lem:homomorphisms}.}
 
\lr{Now suppose that $F$ is
 infinitesimally eventually  isometric. Then $g_n=\id$ for large $n$ by Lemma~\ref{lem:limitfunctions} and $\psi_n=\Psi_n=\id$ by Lemma~\ref{lem:homomorphisms}.  Without loss of generality,
 assume this holds for all $n\geq 0$. 
 Let $K\subset U_0$ be compact, and let $\tilde{K}\subset \DD$ be compact with
 $K\subset \pi_0(\tilde{K})$. Let $v,w\in K$ and let $\tilde{v},\widetilde{w}\in \tilde{K}$ with
 $\pi_0(\tilde{v})=v$, $\pi_0(\widetilde{w})=w$. For $n\geq 0$, 
 set $v_n\defeq f_0^n(v)=\pi_n(\tilde{v})$ and
 define $w_n$ analogously. Then
  \begin{equation}\label{eqn:gammavw} \dist_{U_n}(v_n,w_n) = \min_{\gamma\in\Gamma_n} \dist_{\DD}(\tilde{v},\gamma(\widetilde{w})).\end{equation}
 Recall that $\Gamma_{n}=\psi_n(\Gamma_n) \subset \Gamma_{n+1}$ for every $n$. Since the limit group $\Gamma$ is discrete,
  we have $\Gamma = \bigcup_{n=0}^{\infty} \Gamma_n$. Consider the set $\Gamma_{\tilde{K}}$ of 
  $\gamma\in \Gamma$ such that $\dist_{\DD}(\tilde{K},\gamma(\tilde{K}))\leq \diam_{\DD}(\tilde{K})$, and observe
that the minimiser $\gamma$ in~\eqref{eqn:gammavw} belongs to this set
for all $n$. The set $\Gamma_{\tilde{K}}$ is compact, hence
  finite, so there is $n_0=n_0(K)$ such that $\Gamma_{\tilde{K}}\subset \Gamma_n$ for all $n\geq n_0$. Now
   \[ \dist_{U_n}(v_n,w_n) = \min_{\gamma\in \Gamma_{\tilde{K}}\cap \Gamma_n} \dist_{\DD}(\tilde{v},\gamma(\widetilde{w}))
     = \min_{\gamma\in \Gamma_{\tilde{K}}} \dist_{\DD}(\tilde{v},\gamma(\widetilde{w})). \]
 So the sequence is indeed constant for $n\geq n_0(K)$, and $F$ is locally eventually isometric.}

 \lr{It remains to prove that the sequence $(\delta_n(\lr{z_0}))_{n=0}^{\infty}$ is eventually constant.
  If $\Gamma=\{\id\}$, then all $U_n$ are simply connected, and there is nothing to prove. 
  Otherwise, 
     \[ \delta_n(\lr{z_0}) = \frac{1}{2}\min_{\gamma\in\Gamma_n\setminus\{\id\}} \dist_{\DD}(0,\gamma(0)) = \frac{1}{2}\min_{\gamma\in\Gamma\setminus\{\id\}} \dist_{\DD}(0,\gamma(0)), \] 
  for all sufficiently large $n$.}
\end{proof}

Before finally proving Theorem \ref{thm:main}, we recall the following fact about connectivities of wandering domains (see \cite{Zhe01}*{p. 219}). \lr{Denote the connectivity of a domain $U\subset\Cx$ by $c(U)$.}
\begin{lemma}[Connectivities of Fatou components]\label{lem:zheng}
Let $f$ be a transcendental meromorphic function with
\lr{an orbit $(U_n)_{n=0}^{\infty}$ of wandering Fatou components.
Then one of the following holds.}
\begin{enumerate}[(i)]
    \item \lr{$c(U_n)=\infty$ for all} $n\geq 0$;
    \item there exists $k\in\mathbb{N}$, $k\geq 3$, such that 
    \lr{$c(U_n)=k$ for all large $n$;}
    \item for all large \lr{$n$}, $c(U_n)\in\{1, 2\}$.
\end{enumerate}
\end{lemma}

\begin{proof}[Proof of Theorem~\ref{thm:main}]
\lr{The classification follows from Theorem~\ref{thm:main-nonautonomous}.} We must prove the remaining statements on eventual connectivities. Let $U$ be a wandering domain of a transcendental meromorphic function. By~\cite{Fer22}*{Theorem 1.1}, \lr{$U$ cannot have a finite eventual connectivity greater than $2$, unless $U$ is globally eventually isometric. So there is nothing left to prove if $U$ is infinitesimally contracting (case \ref{item:contracting}), 
or essentially thick and 
infinitesimally semi-contracting (case~\ref{item:semicontracting}). In the globally eventually isometric
 case~\ref{item:globallyisometric}, $f\colon U_n\to U_{n+1}$ is a conformal isomorphism
 for sufficiently large $n$, so $U$ has an eventual connectivity. If $U$ is essentially
 thin and infinitesimally eventually isometric, then by Theorem~\ref{thm:main-nonautonomous},
 it has eventual connectivity $2$.

 So} it only remains to show the following.
\begin{enumerate}[(a)]
    \item If $U$ is locally, but not globally, eventually isometric, then it is infinitely connected.\label{item:locallyisometric}
    \item If $U$ is bimodal, then it \lr{has eventual connectivity $2$ or $\infty$}.%
    \label{item:bimodalpf}
\end{enumerate}

\lr{We first prove \ref{item:locallyisometric}. 
Suppose that $U$ is infinitesimally eventually isometric; then there 
is $n_0$ such that $f\colon U_n\to U_{n+1}$ is a covering map for 
all $n\geq n_0$. If $n\geq n_0$ and $U_n$ is finitely connected,}
only the following possibilities can occur:
\begin{enumerate}[(i)]
    \item $U_n$ is simply connected, $f\colon U_n\to U_{n+1}$ is a universal covering, and $c(U_{n+1}) \geq 1$. If equality holds, $f$ has degree one in $U_n$;\label{item:universalcovering}
    \item $1 < c(U_n) < \infty $, and $c(U_{n+1}) \lr{\leq} c(U_n)$. Indeed, this is an easy consequence of \cite{Bol99}*{Theorem 3}.\label{item:decreasingconnectivity}
\end{enumerate}
If \lr{case~\ref{item:decreasingconnectivity}} holds for all sufficiently large $n$, then $U$ must have an eventual connectivity $k\geq 2$\lr{. If $k\geq 3$, the domain is globally eventually isometric by \cite{Fer22}*{Theorem 1.1 (i)}. If $k=2$, then $f$ is a finite-degree covering map of annuli on sufficiently large $U_n$, and by~\cite{Fer22}*{Theorem~1.1~(ii)}, 
$U$ is either globally eventually isometric, 
or trimodal and hence essentially thin.  On the
other hand, if $c(U_N)=1$ for some $N>n_0$, then
$c(U_n)=1$ and $f\colon U_n \to U_{n+1}$ is a conformal isomorphism
for  $n_0\leq n < N$. Indeed, this follows inductively from the monodromy
theorem. Hence, if $c(U_n)=1$ for infinitely many $n$, this holds for all
$n\geq n_0$, and $U$ is once again globally eventually isometric. Thus, a locally but not globally eventually isometric wandering domain $U$ is eventually infinitely connected by Lemma~\ref{lem:zheng}.} 

For~\ref{item:bimodalpf}, notice that, by Theorem \ref{thm:semi-contracting}, $U_n$ is already at least doubly connected for all sufficiently large $n$. \lr{By Lemma \ref{lem:zheng} and by 
 \cite{Fer22}*{Theorem 1.1}, $U$ must have a finite or infinite eventual connectivity, 
 which can only take the values $2$ or $\infty$.} 
\end{proof}

\section{Examples}\label{sec:ex}
\lr{We now discuss examples that realise the different cases of Theorem~\ref{thm:main}, first considering}
the case of entire functions. \lr{Here 
it is known that every wandering domain has eventual connectivity 
$1$, $2$ or $\infty$~\cite{KS08}*{Theorem~A}.} Examples of \lr{eventually} simply connected globally contracting, globally semi-contracting and globally eventually isometric wandering domains are \lr{constructed} in~\cite{BEFRS19}. An entire function having a wandering domain of eventual connectivity $\infty$ was first constructed by Baker~\cite{Bak63}. 
\lr{Every such wandering domain is infinitesimally semi-contracting~\cite{Fer22}*{Theorem~1.2}.} A wandering domain with eventual connectivity $2$ was constructed by Kisaka and Shishikura~\cite{KS08}\lr{; every such wandering domain is infinitesimally
eventually isometric by~\cite{KS08}*{Theorem~A}}. This \lr{shows that all cases of Theorem~\ref{thm:main} except~\ref{item:nongloballyisometric} can be realised
by entire functions. In this section, we will prove Theorem~\ref{thm:ex}, realising~\ref{item:nongloballyisometric} by a meromorphic
function.} 

\lr{Before we do so, we briefly discuss meromorphic functions having wandering domains of the types from Theorem~\ref{thm:main} with various eventual connectivities
(or none), beyond those that are realised by entire functions.}
Globally eventually isometric wandering domains of any eventual connectivity $\geq 2$ are constructed in~\cite{BKL90}; their existence also follows from Theorem~\ref{thm:magic} or~\cite{MartiPeteRempeWaterman}*{Theorem~1.3}. \lr{Hence all cases 
of~\ref{item:globallyisometric} can occur.}
      
Examples of essentially thick globally contracting or semi-contracting wandering domains with eventual connectivity $\infty$ are given in \cite{RS08}*{Example 1} and \cite{Fer24}*{Theorem 1.2}, respectively.
     
\lr{An example} of wandering domains without an eventual connectivity \lr{is} constructed in \cite{Fer22}*{Theorem 1.3}. \lr{Moru\'s~\cite{morus2026wandering} showed that it is
possible to modify this construction to obtain 
contracting wandering domains without eventual connectivity; moreover, they
can be chosen to be either essentially thick or essentially thin. 
Recall from Lemma~\ref{lem:zheng} that if $U_0$ is a wandering domain of a meromorphic function that has no 
eventual connectivity and $(U_n)_{n=0}^{\infty}$ is its orbit, then $U_n$ is simply connected for infinitely many $n$. 
In particular, Moru\'s's construction also gives an example of a contracting wandering 
domain $U$ where $\liminf \delta_n(z)=0$ and $\limsup \delta_n(z) = \infty$.
The analysis also suggests that a modification can yield essentially thick and infinitesimally semi-contracting examples
without eventual connectivity; this is the subject of further investigation.}

\lr{Let us briefly discuss the case of eventual connectivity $2$ 
 in~\ref{item:contracting},~\ref{item:semicontracting} and~\ref{item:bimodal}. 
  It seems plausible that such wandering domains exist, but as far as we know,
 they are not present in the existing literature. 
 We note that, by~\cite{Fer22}*{Theorem~1.1}, in such an example  
 the degree of $f|_{U_n}$ must be infinite (and hence $U_n$ must be unbounded)
 for infinitely many $n$.}

It remains to prove Theorem~\ref{thm:ex}. To do so, we construct a suitable model map to which we can apply Theorem~\ref{thm:magic}.

 \begin{prop}[A locally but not globally eventually isometric model]\label{prop:model}
   There exist sequences $(b_n)_{n=0}^{\infty}$
      of degree $2$ Blaschke products $\DD\to\DD$
     and $(U_n)_{n=0}^{\infty}$ of subdomains of $\DD$ with $0\in U_n$ for all $n\geq 0$ such that:
      \begin{enumerate}
        \item $b_n|_{U_n}\colon U_n\to U_{n+1}$ is a degree two covering map and $b_n(0)=0$ for all $n\geq 0$;
        \item there is an increasing sequence $(r_n)_{n=0}^{\infty}$ with $r_n\nearrow 1$ 
         such that  $b_n|_{D(0,r_n)}$ is injective, for all $n\geq 0$.
      \end{enumerate}
 \end{prop}
 \begin{proof}
   For $0<a<1$, define
     \[ b_a(z) \defeq z\cdot\frac{z + a}{1 + az}. \]
   The maps $b_n$ in the proposition will be of the form $b_n = b_{a_n}$, where
   $a_n \nearrow 1$ is a suitable sequence of positive numbers.

   The functions $b_a\colon\DD\to\DD$ have degree two, fix the origin,
      and have a single critical point at
\[ c_a = \frac{-1 + \sqrt{1 - a^2}}{a}. \]
The critical point approaches $-1$ as $a\to 1$, as does its value $v_a \defeq b_a(c_a)$.
 Clearly $b_a$ converges to the identity as $a\to 1$; in particular,
 for any $r<1$, $b_a$ is injective on $D(0,r)$ for $a$ sufficiently close to $1$. 
 \lr{(}In fact, it is easy to see that $b_a$ is injective on
 $D(0,\lvert c_a\rvert)$.\lr{)}

We now construct the numbers $a_n$ inductively. As part of the construction, we also define the increasing sequence $r_n$, as well as compact subsets   $A_k^n \subset \DD$, for $0\leq k\leq n+1$, for $n\geq 0$, each of which is a finite union of \lr{closed} 
topological discs. These are chosen to have the following properties for all $n\geq 0$.
  \begin{enumerate}[(i)]
    \item $b_n$ is injective on $D(0,r_n)$.\label{item:injective}
    \item $0\notin A_k^n$ for all $k \leq n+1$ and
        \lr{$c_k\defeq c_{a_k} \in A_k^n$} for all $k\leq n$.\label{item:0andcritical}
    \item $b_k^{-1}(A_{k+1}^n) = A_{k}^n$ for $0\leq k\leq n$.\label{item:preimage}
    \item If $n\geq 1$, every connected component of $A_k^{n-1}$ is also
        a connected component of $A_k^{n}$ and
         $\lr{(A_k^{n}\setminus A_k^{n-1})}\cap D(0,r_n)=\emptyset$
        for $0\leq k \leq n$. \label{item:newdomainsnearboundary}
 \end{enumerate}

Suppose that $m\geq 0$ is such \lr{that}, for all $n<m$, the numbers $a_n$ and sets $A_k^n$ have been defined with the above properties.
    (This inductive hypothesis is empty
  for $m=0$.) If $m=0$, set $r_0 \defeq 1/2$; otherwise,
      choose $r_m\in (r_{m-1},1)$ such that $r_m > 1-1/(m+1)$ and 
     $A_m^{m-1}\subset D(0,r_m)$. Then choose $a_m$ close enough to $1$ \lr{to ensure
      that $c_m\notin \overline{D(0,r_m)}$ and} $b_m$ is injective on $D(0,r_m)$\lr{. In particular,}~\ref{item:injective} holds for $n=m$.

Since $c_m\notin \lr{\overline{D(0,r_m)}}$, we have $v_m\defeq b_m(c_m)\notin \lr{b_m(\overline{D(0,r_m)})}$ (note that $b_m^{-1}(v_m)=\{c_m\}$). Choose $\eps>0$ such that $D(v_m,\eps)\cap b_m(D(0,r_m)) = \emptyset$. Define
\begin{equation}\label{eqn:Am1}
          A_{m+1}^m \defeq \begin{cases}                 
          \overline{D(v_m,\eps)} \cup  b_m(A_m^{m-1}) &\text{if }m>0 \\
          \overline{D(v_m,\eps)} &\text{if }m=0\end{cases}
\end{equation}
and define $A_k^m$ recursively according to~\ref{item:preimage} for
     $k=m,m-1,\dots,0$.

    \lr{If $m\geq 1$, let} us consider the connected components of $A_m^m = \lr{b_m}^{-1}(A_{m+1}^m)$. One such component is given by \lr{$b_m^{-1}(\overline{D(v_m,\eps)})$}, which is disjoint from $D(0,r_m)$ by our choice of $\eps$. \lr{Every} other component  of $A_m^m$  is a connected component of $b_m^{-1}(b_m(A))$, where $A$ is a connected component of $A_m^{m-1}\subset D(0,r_m)$. Since $b_m$ is injective on $D(0, r_m)$, there are two such components: $A$ itself, and a second component $A' \subset \DD\setminus D(0,r_m)$. This establishes~\ref{item:newdomainsnearboundary} for $n=k=m$. For $k<m$~\ref{item:newdomainsnearboundary} follows inductively from~\ref{item:preimage} and the Schwarz lemma.

    Furthermore, \lr{$0\notin \overline{D(v_m,\eps)}$ and, if $m\geq 1$, then $0\notin b_m(A_m^{m-1})$} since $b_m$ is injective on $D(0,r_m)\supset A_m^{m-1}$ and
    $0\notin A_m^{m-1}$ by the inductive hypothesis\lr{. So} $0\notin A_{m+1}^m$ by~\eqref{eqn:Am1}. 
    By~\ref{item:preimage}, we conclude inductively that $0\notin A_k^m$ for $k<m+1$
    (recall that all the functions $b_k$ fix $0$). Moreover, $c_m\in A_m^m$, and by the inductive hypothesis and~\ref{item:newdomainsnearboundary} we also have $c_k\in A_k^m$ for $k<m$. This proves~\ref{item:0andcritical} and completes the recursive construction.

    For $n\geq 0$,  we define
       \[ U_n \defeq \DD \setminus \bigcup_{k=n}^{\infty} A_n^k. \]
    Observe that, by~\ref{item:newdomainsnearboundary}, $U_n$ is a subdomain of $\DD$, and furthermore $0\in U_n$, $c_n\notin U_n$, and $b_n^{-1}(U_{n+1}) = U_n$. It follows that $b_n\colon U_n\to U_{n+1}$ is indeed a degree two covering map. The proof of the proposition is complete.
\end{proof}

\begin{proof}[Proof of Theorem~\ref{thm:ex}]
 Fix the sequences $(b_n)_{n=0}^{\infty}$ and $(U_n)_{n=0}^{\infty}$ from Proposition~\ref{prop:model} and  define
\[ T_n(z) \defeq z + 4n, \quad  X_n \defeq \overline{T_n(U_n)}\quad\text{and}\quad
     \phi(z) \defeq T_{n+1}( b_n((T_n)^{-1}(z)))\ \text{for $z\in X_n$}.\]

 Then $\phi$ is holomorphic on a neighbourhood of $X\defeq \bigcup_{j=0}^{\infty} \lr{X_j}$.
By Theorem \ref{thm:magic}, there exists a meromorphic function $f\colon\Cx\to\Chat$
 and a diffeomorphism $\theta\colon\Cx\to\Cx$ that is conformal on $\intr(X)$ 
 such that 
$V_n\defeq \theta(T_n(U_n))$ is a
 Fatou component of $f$ for every $n$ and 
     $\theta\circ \lr{\phi} = f\circ \theta$ on $X$.  
    It follows that $f\colon V_n\to \lr{V}_{n+1}$ is an unbranched covering of degree two for all $n\geq 0$. In particular, $(V_n)_{n=0}^{\infty}$ is an orbit
    of infinitesimally \lr{eventually} isometric wandering domains that is not globally eventually isometric.
    Since the $V_n$ are infinitely connected, it follows from Theorem~\ref{thm:main} that 
    this orbit of wandering domains is locally eventually isometric.
    \end{proof}

\begin{rmk}
    It is also not difficult to show the final claim in the proof
     (that the orbit of wandering domains is locally eventually isometric) directly, without referring to Theorem~\ref{thm:main}.
\end{rmk}

\bibliography{ref}

\end{document}